\documentclass[a4paper,10pt]{article}
\pdfoutput=1
\usepackage[utf8]{inputenc}
\usepackage[final]{graphicx}
\usepackage{amsmath}				
\usepackage{amssymb}				
\usepackage{euscript}
\usepackage[amsmath,thmmarks]{ntheorem}	
\usepackage[all,2cell]{xy}
\usepackage[colorlinks=true]{hyperref}
\usepackage{stmaryrd}
\usepackage{wasysym}

\usepackage{bussproofs}
\EnableBpAbbreviations
\def\fCenter{\; \vdash\; }

\newcommand{\catset}{\ensuremath{\mathbf{Set}}}

\newcommand{\qcb}{{\mathcal{B}}}
\newcommand{\qcc}{{\mathcal{C}}}
\newcommand{\qcd}{{\mathcal{D}}}

\newcommand{\flc}{\ensuremath{\mathbb{C}}}
\newcommand{\fld}{\ensuremath{\mathbb{D}}}
\newcommand{\fle}{\ensuremath{\mathbb{E}}}

\newcommand{\flx}{\ensuremath{\mathbb{X}}}
\newcommand{\fly}{\ensuremath{\mathbb{Y}}}

\newcommand{\tope}{\ensuremath{\mathcal{E}}}
\newcommand{\topf}{\ensuremath{\mathcal{F}}}

\newcommand{\bbc}{\mathbb{C}}

\newcommand{\catcat}{\ensuremath{\mathbf{Cat}}}
\newcommand{\catfp}{\ensuremath{\mathbf{Fp}}}
\newcommand{\cattwocat}{\ensuremath{\mathbf{2}\text{-}\mathfrak{Cat}}}

\newcommand{\cattrip}{\ensuremath{{\mathbf{Trip}}}}
\newcommand{\cattop}{\ensuremath{{\mathbf{Top}}}}
\newcommand{\catqtop}{\ensuremath{{\mathbf{QTop}}}}

\newcommand{\twocata}{\ensuremath{\mathbf{A}}}
\newcommand{\twocatb}{\ensuremath{\mathbf{B}}}
\newcommand{\twocatc}{\ensuremath{\mathbf{C}}}
\newcommand{\twocatd}{\ensuremath{\mathbf{D}}}
\newcommand{\twocate}{\ensuremath{\mathbf{E}}}
\newcommand{\twocatx}{\ensuremath{\mathbf{X}}}

\newcommand{\pcaa}{\ensuremath{\mathcal{A}}}
\newcommand{\pcaas}{\ensuremath{\mathcal{A}_\#}}

\newcommand{\trip}{\ensuremath{\EuScript{P}}}
\newcommand{\triq}{\ensuremath{\EuScript{Q}}}

\newcommand{\rtas}{\ensuremath{\mathbf{RT}(\pcaas)}}
\newcommand{\rtaas}{\ensuremath{\mathbf{RT}(\pcaa, \pcaas)}}
\newcommand{\rtaaso}{\ensuremath{\mathbf{RT}_o(\pcaa, \pcaas)}}
\newcommand{\rtaasc}{\ensuremath{\mathbf{RT}_c(\pcaa, \pcaas)}}

\newcommand{\true}{\ensuremath{\mathsf{true}}}
\newcommand{\false}{\ensuremath{\mathsf{false}}}

\newcommand{\vtp}{\mathopen:\,}				

\newcommand{\ttp}{\mathrel:}				
\newcommand{\csep}{\mathrel|\,}				
\newcommand{\tripower}{{}\mathfrak{P}}			
\newcommand{\op}{\mathsf{op}}				
\newcommand{\prop}{\ensuremath{\mathsf{Prop}}}		
\newcommand{\triptr}{\ensuremath{\mathsf{tr}}}		
\newcommand{\toptr}{\ensuremath{\mathsf{t}}}		

\newcommand{\id}{\mathrm{id}}				
\newcommand{\qdot}{\,.\hspace{.3mm}}				
\newcommand{\objects}[1]{#1}

\newcommand{\funs}{{\boldsymbol{S}}}			
\newcommand{\funt}{{\boldsymbol{T}}}			
\newcommand{\funu}{{\boldsymbol{U}}}
\newcommand{\funf}{{\boldsymbol{F}}}
\newcommand{\fung}{{\boldsymbol{G}}}
\newcommand{\eq}{\mathsf{eq}}

\newcommand{\coloneq}{\mathrel{\mathop:}=}

\newcommand{\sem}[1]{{\llbracket #1 \rrbracket}}
\newcommand{\semp}[1]{{\llbracket #1 \rrbracket}}

\newcommand{\commacat}[2]{#1\mathopen\downarrow#2}

\newcommand{\spec}{{\mathfrak{Spec}}}

\newcommand{\supp}{\mathsf{supp}}
\newcommand{\bool}{\mathbb{B}}
\newdir{ >}{{}*!/-10pt/@{>}}
\newdir{|>}{%
!/4.5pt/@{|}*:(1,-.2)@^{>}*:(1,+.2)@_{>}}
\newdir{  |>}{%
!/-8.5pt/:(1,-.2)@^{>}*!/-8.5pt/:(1,+.2)@_{>}*!/-4pt/@{|}}

\newcommand{\effi}{{(F,\Phi)}}

\newcommand{\ggamma}{{(G,\Gamma)}}

\newcommand{\cro}{{(C,\rho)}}

\newcommand{\dsi}{{(D,\sigma)}}

\newcommand{\predeq}{\mathclose=}

\newcommand{\tttc}{tripos-to-topos construction}

\newcommand{\reg}{\ar@{|->}}
\newcommand{\mono}{\ar@{ >->}}
\newcommand{\cocov}{\ar@{ |>->}} 
\newcommand{\depi}{\ar@{->>}}
\newcommand{\cov}{\ar@{-|>}}
\newcommand{\monepi}{\ar@{ >->>}}
\newcommand{\dashed}{\ar@{-->}}

\newcommand{\dashprof}{\ar@{-->}|-*=0@{|}}
\newcommand{\ppair}[3]{{\ar@<3pt>[#1]^{#2}\ar@<-3pt>[#1]_{#3}}}
\newcommand{\nohead}{\ar@{-}}

\renewcommand{\to}{\rightarrow}
\newcommand{\emono}{\rightarrowtail}
\newcommand{\eepi}{\twoheadrightarrow}
\newcommand{\eiso}{\stackrel{\cong}{\longrightarrow}}
\newcommand{\ecocov}{\mathrel{\raisebox{.26mm}{$\scriptstyle\triangleright$}\hspace{-1.5mm}\rightarrow}}

\newcommand{\emonepi}{\rightarrowtail\hspace{-2.9mm}\twoheadrightarrow}
\newcommand{\ecov}{\rightarrowtriangle}
\newcommand{\ecovleft}{\leftarrowtriangle}

\newcommand{\ve}{\varepsilon}
\newcommand{\sdi}[1]{\vcenter{\hbox{\includegraphics{#1.pdf}}}}

\newcommand{\qtext}[1]{\quad\text{#1}\quad}

\newcommand{\regular}{regular}
\newcommand{\subc}{\operatorname{sub}_c}
\newcommand{\vertor}{\mathrel|}

\newcommand{\imp}{\Rightarrow}

\newcommand{\cxymatrix}[1]{\vcenter{\xymatrix{#1}}}

\newcommand{\funcoc}{\mathbf{coc}}
\newcommand{\pereqv}{\mathsf{eqv}}

\newcommand{\bftr}{{\funf\trip}}
\newcommand{\bsqc}{{\funs\qcc}}
\newcommand{\fsc}{\funf\funs\qcc}
\newcommand{\fsd}{\funf\funs\qcd}

\newcommand{\fsfp}{\funf\funs\funf\trip}

\newcommand{\wcol}{\; : \;}

\newcommand{\qcol}{\quad : \quad}
\newcommand{\qarr}{\quad \to \quad}
\newcommand{\pseta}{\boldsymbol{\eta}}
\newcommand{\psve}{\boldsymbol{\epsilon}}
\newcommand{\crefl}[1]{{\overline{#1}}}
\newcommand{\utc}{{\funu\funt\qcc}}

\newcommand{\defpred}{\quad\coloneq\quad}
\newcommand{\idefpred}{\;\coloneq\;}
\newcommand{\defcon}{\quad\equiv\quad}

\newcommand{\tworeg}{geometric}
\newcommand{\sigs}{{\mathbf{S}}}
\newcommand{\sigf}{{\mathbf{F}}}
\newcommand{\sigr}{{\mathbf{R}}}
\newcommand{\hots}{\mathcal{T}} 
\newcommand{\ptype}{\mathbf{P}}
\newcommand{\hotsig}{\hots(\sigs)}
\newcommand{\hotsigp}{\hots(\sigs^\trip)}
\newcommand{\fots}{\mathcal{T}_0}

\newcommand{\holt}[1]{\mathcal{L}({#1})}
\newcommand{\folt}[1]{\mathcal{L}_0({#1})}

\newcommand{\unary}[3]{
\AxiomC{$#1$}
\RightLabel{#2}
\UnaryInfC{$#3$}
\DisplayProof
}

\newcommand{\binary}[4]{
\AxiomC{$#1$}
\AxiomC{$#2$}
\RightLabel{#3}
\BinaryInfC{$#4$}
\DisplayProof
}

\newcommand{\ternary}[5]{
\AxiomC{$#1$}
\AxiomC{$#2$}
\AxiomC{$#3$}
\RightLabel{#4}
\TrinaryInfC{$#5$}
\DisplayProof
}

\newcommand{\into}[2]{#1\{#2\}}
\newcommand{\perprod}{\Join}

\newcommand{\ilbracks}[1]{\lrbracks{#1}}
\newcommand{\lrbracks}[1]{\bigl(#1\bigr)}
\newcommand{\nhom}{\mathbf{hom}}
\newcommand{\preq}{\mathfrak{Spec}}
\newcommand{\oplax}{\mathfrak{Oplax}}
\newcommand{\comp}{\operatorname{comp}}

\newcommand{\pseudo}{\mathfrak{Pseudo}}

\newcommand{\bec}{bicategory enriched category}

\newcommand{\famf}{\mathrm{fam}}
\newcommand{\Famf}{\mathrm{Fam}}
\newcommand{\adj}{\dashv}

\newcommand{\emar}{\ar@{}}

\newcommand{\adjr}[2]{
\ar@/_6pt/[r]_{#1}
\ar@{}[r]|\top
\ar@{<-}@/^6pt/[r]^{#2}
}
\newcommand{\adjl}[2]{
\ar@/^6pt/[l]^{#1}
\ar@{}[l]|\top
\ar@{<-}@/_6pt/[l]_{#2}
}
\newcommand{\adjd}[2]{
\ar@/_6pt/[d]_{#1}
\ar@{}[d]|\dashv
\ar@{<-}@/^6pt/[d]^{#2}
}
\newcommand{\adju}[2]{
\ar@/^6pt/[u]^{#1}
\ar@{}[u]|\dashv
\ar@{<-}@/_6pt/[u]_{#2}
}

\theoremnumbering{arabic}
\theoremstyle{plain}
\theoremsymbol{}
\theorembodyfont{\itshape}
\theoremheaderfont{\normalfont\bfseries}
\theoremseparator{}
\newtheorem{theorem}{Theorem}[section]

\newtheorem{lemma}[theorem]{Lemma}
\newtheorem{corollary}[theorem]{Corollary}

\theorembodyfont{\upshape}
\theoremsymbol{\ensuremath{\diamondsuit}}
\newtheorem{definition}[theorem]{Definition}
\newtheorem{definitions}[theorem]{Definitions}
\theoremsymbol{\ensuremath{\diamondsuit}}
\newtheorem{example}[theorem]{Example}

\theorembodyfont{\normalfont}

\newtheorem{convention}[theorem]{Convention}

\newtheorem{remark}[theorem]{Remark}

\theoremstyle{nonumberplain}
\theoremsymbol{\ensuremath{_\blacksquare}}
\theoremheaderfont{\itshape}
\newtheorem{proof}{Proof.}
\qedsymbol{\ensuremath{_\blacksquare}}

\numberwithin{equation}{section}		

\DeclareMathAlphabet{\mathpzc}{OT1}{pzc}{m}{it}
\allowdisplaybreaks
\UseTwocells

\title{A 2-Categorical Analysis of the \\ Tripos-to-Topos Construction}
\author{Jonas Frey}

\begin{document}

\maketitle

\begin{abstract}
We characterize the tripos-to-topos construction of Hyland, Johnstone and
Pitts as a biadjunction in a bicategory enriched category of equipment-like
structures. These abstract concepts are necessary to handle the presence of
\emph{oplax} constructs --- the construction is only \emph{oplax} functorial
on certain classes of cartesian functors between triposes. 

A by-product of our analysis is the decomposition of the \tttc{} into two steps,
the intermediate step being a weakened version of quasitoposes.
\end{abstract}

\pagebreak

\tableofcontents

\pagebreak

\section{Introduction}\label{sec_intro}

Triposes were introduced by Hyland, Johnstone and Pitts~\cite{hjp80} as a
framework which 
enables to generalize the construction of the category of sheaves
on a locale
(complete
Heyting algebra).
Their motivating observations were that
\begin{itemize}
 \item the alternative description of sheaves on a locale $A$ as `$A$-valued
sets' which was independently introduced by Higgs~\cite{higgs1973category},
and by Fourman and Scott~\cite{fourman1979sheaves}, really only depends on the
fibered
poset $\famf(A):\Famf(A)\to\catset$ (the \emph{family fibration} of $A$, see
\cite[Definition~1.2.1]{jacobs2001categorical}), and
\item
Kleene's number realizability gives rise to a fibration of preorders on
which Higgs, Fourman and Scott's construction is
defined and yields a
topos as well!
\end{itemize}
These observations lead to the question which properties of the fibration are
really needed to allow the construction of toposes, and the definition of
tripos gives sufficient conditions (but still stronger than necessary ones, as
Pitts points out in~\cite{pitts2002}).

The topos that arises from the tripos associated to Kleene's number
realizability is Hyland's \emph{effective topos}, its introduction marks the
starting point a whole new research field: \emph{categorical
realizability}\footnote{For an introduction to this field, we refer to Jaap
van~Oosten's recent textbook~\cite{vanoosten2008realizability}.}.

The cross-fertilization between realizability and topos theory/category theory
has proven fruitful to categorical logicians and topos theorists on the one
hand, since it provides interesting examples of non-Grothendieck toposes, and
to realizability on the other hand, since it brought new categorical tools, and
a more `semantic' way of thinking to a field which had traditionally been
frightening due to its high amount of syntactical formalism. 
The new perspective on realizability lead to the discovery of new, `global'
connections between different notions of realizability, making use of
geometrically motivated topos theoretic concepts such as \emph{geometric
morphism} and \emph{subtopos} (associated to a Lawvere-Tierney topology). As
examples, we mention
\begin{itemize}
\item 
Awodey, Birkedal and Scott's work~\cite{awodey2002local}, where a
\emph{local%
\footnote{
A geometric morphism is called \emph{local} if its unit is invertible and the
direct image part has a further right adjoint.}
geometric
morphism}
\[
\Delta\adj\Gamma:\rtaaso\to\rtas\adj\nabla
\]
is exhibited between the \emph{relative} realizability topos $\rtaaso$ induced
by an
inclusion $\pcaas\subset\pcaa$ of partial combinatory algebras and the
realizability topos $\rtas$, and
\item
Birkedal and van~Oosten's paper~\cite{birkoost} which
describes how the relative realizability topos $\rtaaso$ and the
\emph{modified%
\footnote{In light of~\cite{lietz2002impredicativity}, it is arguable  whether
modified (relative) realizability should really be viewed as a topos or if
$\rtaaso$ rather represents something else (since modified realizability has a
typed notion of realizer), but this doesn't bother us here.
}
 relative} realizability topos $\rtaasc$ can be viewed as open and closed
complementary subtoposes 
\[
\rtaasc\hookrightarrow\rtaas\hookleftarrow\rtaaso
\]
of a larger topos $\rtaas$\footnote{%
See e.g.~\cite[A4.5]{elephant1} for the definition of open and closed
subtoposes.}.
\end{itemize}
Abstractly, geometric morphisms and subtoposes are just adjunctions and
idempotent monads in the 2-category of toposes and cartesian functors, and we
have analogous concepts in an appropriate 2-category of triposes.
Furthermore, the geometric morphisms and subtoposes in the previous
examples are induced by analogous constructs between the corresponding
triposes. 

It turns out that it is much easier to make
calculations on the level of triposes than on the level of toposes, to the
extent that we would like to systematically reduce questions about functors
between tripos-induced toposes to questions about morphisms between the
corresponding triposes. But in order to do this, we need an abstract (i.e.,
universal) characterization of the construction which maps triposes to toposes
and morphisms between triposes to functors between toposes. 
\emph{This is the
motivation and the objective of the present work.
}
\medskip

The question for a universal characterization of the \tttc{} is not a new one.
Already in 2002, Pitts wrote~\cite{pitts2002}:
\begin{quote}
The construction itself can be seen as the
universal solution to the problem of realizing the predicates of a first order
hyperdoctrine as
subobjects in a logos with effective equivalence relations.
\end{quote}
In a more recent, unpublished work~\cite{rosmai08}, Rosolini and Maietti
decompose the \tttc{} into a succession of fibrational completions.

These approaches answer the question for a universal characterization, but
are not adequate as a framework for the above examples, since they (albeit
implicitly) take place in the 2-categories of triposes and \emph{regular}
tripos morphisms (that is fibered functors that commute with $\wedge$ and
$\exists$), and toposes and \emph{regular} functors. In order to talk about
arbitrary geometric morphisms and subtoposes/sub-triposes, this is too
restrictive --- we want to talk about functors and morphisms which only
preserve finite limits and finite meets, respectively.

Already in \cite{hjp80}, it was observed that it is
possible to construct functors between toposes from tripos morphisms that
merely commute with finite limits, but the abandonment of regularity leads to
complications which require more sophisticated 2-dimensional techniques, as the
following example demonstrates.

Let $\bool=\{\true, \false\}$ be the locale of booleans, with
$\false\leq\true$. Then
$\famf(\bool)$ and $\famf(\bool\times\bool)$ are triposes, and the
induced toposes are
equivalent to $\catset$ and $\catset\times\catset$, respectively. Between the
locales we consider
the meet-preserving maps
\[\delta=\langle\id,\id\rangle:\bool\to\bool\times\bool\qquad\mbox{and}
\qquad\wedge:\bool\times\bool\to\bool\]
These maps give rise to tripos morphisms
\[
\xymatrix@1@C+5mm{\famf(\bool)\ar[r]^-{\famf(\delta)}&
\famf(\bool\times\bool)\ar[r]^-{\famf(\wedge)} & \famf(\bool)},
\]
which in turn give rise to functors
which happen to be the familiar
\[
\xymatrix@1@C+7mm{\catset\ar[r]^-{\Delta=\langle\id,\id\rangle}&
\catset\times\catset\ar[r]^-{(-\times-)} & \catset},
\]
Forming the composition of the maps, we get $\wedge\circ\delta=\id_\bool$ and
this gives rise to
the
identity functor. Therefore we obtain a \emph{non-invertible} constraint cell
\[
\xymatrix@C-6mm{
\catset\ar[rd]_\Delta\ar[rr]^\id|{}="1"\xtwocell[rr]{}<\omit>{<3>\eta}
&&\catset\\
&\catset\times\catset\ar[ur]_\times
}
\]
where $\eta_I=\delta_I:I\to I\times I$ is the unit of the adjunction
$\Delta\dashv (-\times-)$.

This means that the \tttc{} does not commute with composition of tripos
morphisms (not even up to isomorphism), and hence it can not be a 2-functor or
a pseudofunctor. The best that we can hope for is for it to be \emph{oplax
functorial}, which means that is commutes with identities and composition up
to non-invertible 2-cell.

This turns out to be a major obstacle, since we would like to characterize the
construction as a kind of left biadjoint to the construction that assigns its
subobject fibration to a topos. Unfortunately, it is known that (op)lax
functors are
very badly behaved --- horizontal compositions like $F\eta$ for example are
simply not definable for
(op)lax functors $F$ and transformations $\eta$ (see
diagram~\eqref{eq_noncomposable} in Section~\ref{suse_preequs}). That means in
particular that we can not
transfer the algebraic definition of biadjunctions by unit, counit and
modifications for the triangle equalities to the (op)lax world in a
straightforward way.

To overcome these problems, we identify a class of oplax functors
and transformations that compose well, more formally we define a three
dimensional category of 2-categories with additional structure (so-called
\emph{pre-equipments}), and corresponding oplax functors and transformations in
which there is an internal notion of biadjunction that fits our purposes.
Ideas like these have come up at different places in the literature already,
in particular in the context of double categories. The most general
treatment can be found in Dominic Verity's thesis~\cite{verity_enriched} on
which we rely heavily.
 
The \tttc{} can then be described as a biadjunction between the pre-equipments
$\cattrip$ of triposes, and the
pre-equipment $\cattop$ of toposes. Trying to
find a comprehensible description of the left adjoint, I observed that the
construction naturally factors through a third pre-equipment --- the
\emph{q-toposes} (suggestions for a better name are welcome), which are a
generalization of quasitoposes where not all finite colimits are required.
The q-toposes can be viewed as giving an official status to the so-called
\emph{weakly complete objects} that already occur in the original
paper~\cite{hjp80} by Hyland, Johnstone and Pitts.

\subsection{Overview of the article}

The article is divided in four main sections.

Section~\ref{sec_pre-equipments} provides the category theoretical background.
We review Verity's notion of \emph{bicategory enriched category}, and
define the \bec{} $\spec$ of pre-equipments, special functors,
special transformations and modifications. Finally, we introduce \emph{special
biadjunctions} which are just biadjunctions in $\spec$ and which we use in
the sequel for our characterization of the \tttc.

In Section~\ref{sec_triposes} we define triposes, define their internal
language, and explain how they form a pre-equipment.

In Section~\ref{sec_q_toposes} we introduce q-toposes, define the
pre-equipment $\catqtop$ of q-toposes, explain how to
interpret higher order intuitionistic logic in a q-topos and prove that the
\emph{coarse objects} in a q-topos form a topos.

Finally, in Section~\ref{sec_tttc}, we give a detailed exposition of the
special adjunctions $\funf\dashv\funs:\catqtop\to\cattrip$ between triposes
and q-toposes and $\funt\dashv\funu:\cattop\to\catqtop$ between q-toposes and
toposes. These special biadjunctions form our characterization and
decomposition of the \tttc.

\subsection{Conventions, preliminaries}\label{suse_conv}

\noindent\textbf{Notation, terminology}

In 2-categories, and in bicategory enriched categories as introduced
in Section~\ref{sec_pre-equipments}, we will write `$\id$' for identities in
all
dimensions, usually appropriately subscripted.

We will normally write $A\in \flc$ instead of $A\in\operatorname{obj}(\flc)$ to
mean that $A$ is an object of a category $\flc$.

\medskip

\noindent\textbf{Strict, strong, lax and oplax}

We consider different kinds of functors and transformations between
2-categories, and we will use the adjectives \emph{strict, strong, lax and
oplax} to specify whether they have identity-, isomorphic or directed
constraint cells. We also refer to strong functors as pseudofunctors. 

For an oplax
functor $F:\twocata\to\twocatb$ and $A\xrightarrow{f}B\xrightarrow{g}C$ in
$\twocata$, the direction of constraint cells is $F(gf)\to Fg\; Ff$ and
$F\id\to \id$, and for an oplax transformation $\eta:F\to
G:\twocata\to\twocatb$, the direction of constrains is
$\eta_B\;Ff\to Gf\;\eta_A$. For lax functors and transformation, the direction
of constraints is the opposite (in parts of the literature, the meaning of
lax and oplax is exchanged for transformations). 

For definitions of strict, strong, and (op)lax functors, and transformations,
see for example Leinster's~\cite{leinster1998basic} (Note that Leinster uses
the traditional terms morphism and homomorphism for lax and strong functors). 

\medskip

\noindent\textbf{Size issues}

We will assemble possible large categories into 2-categories, and then
assemble theses 2-categories into a three dimensional category. Formally, we
need several Grothendieck universes to do this, but since we do not use
concepts where relative sizes are important (such as local smallness) this
does not pose problems (once we accept the existence of Grothendieck
universes), and we will comment no more on that.

A related issue is that we will talk about presheaves of subobjects and
representable presheaves, concepts which only make sense if the involved
categories are well powered and locally small, respectively. While we want to
avoid to appeal to local smallness and well-poweredness since they refer to
relative sizes, we can
always assume the existence of a universe which makes the involved categories
\emph{globally small} (and thus of course well powered).

\medskip

\noindent
\textbf{String diagrams}

In addition to pasting diagrams, we will use string diagrams for
2-categorical reasoning since they are usually more concise and more
importantly they make the structure of the calculations more apparent.
Many string diagrammatic calculi have been developed and investigated for
different kinds of monoidal categories (see~\cite{selinger2009survey} for an
overview), but we will only use the most basic variant, which exists in a
version for monoidal category and a version for bicategories (which we will
use). This basic version for bicategories is presented e.g.\ in
\cite{street1995low}, and since it is very easy, we explain it again here.

The basic idea of string diagrams is that they are a kind of dual graphs of
pasting diagrams, as visualized in the following example for a composite of two
2-cells $\alpha,\beta$ in some generic 2-category $\twocata$.
\[
\vcenter{\xymatrix@-5mm{
& A\ar[ldd]_h\xtwocell[rr(1.1)]{}<\omit>{^<3>\beta} && C\ar[ll]_k\ar[ldd]^f \\
\\
B\xtwocell[rr(1.1)]{}<\omit>{^<-3>\alpha} && B\ar[luu]|g\ar[ll]^\id
}}
\quad\text{becomes}\quad
\sdi{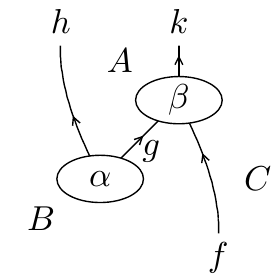}
\]
In the pasting diagram on the left, objects are vertices, morphisms are edges
and 2-cells can be viewed as faces. In the string diagram on the right, the
2-cells are nodes, the morphisms are edges again, but orthogonal to the edges
in the pasting diagram, if we place the diagrams one upon the other; and the
faces correspond to objects of the 2-category. Observe that the 2-cell
$\alpha:\id_B\to hg$ is drawn as a node with zero inputs and 2 outputs; this is
because in this context we think of $\id_B$ as the composite of the empty list
of
1-cells. Normally we omit much of the typing information, i.e. the labels of
lines and faces, because it clutters the diagrams and can be easily inferred
from the context. Moreover, the orientation of our diagrams is always bottom-up
and right-left, and we omit the redundant arrows on the edges as well.

 We are not too much concerned about formal properties
of the calculus of string diagrams itself, our use of them is heuristic rather
than formalistic. In a sense, we view string diagrams as shorthands for more
rigorous symbolic computations, which can always be reconstructed from them on
demand (and for the size of the diagrams that we are using, this is not only a
theoretical but a practical possibility --- with a bit of practice you can even
read the associated pasting diagram ``between the lines'').

\medskip

\noindent
\textbf{Existence of structures versus chosen
structures}

By a finite limit category, do we mean a category such that for any finite
diagram there exists a limiting cone, or a category equipped with a specific
choice of such limiting cones? --- Normally, one tends to be rather ambiguous
about that, after all suitable choice principles always allow us to postulate
an explicit family of limiting cones, even if we only assumed mere existence
before. When we assemble our categories into 2-categories, however, we have to
be more precise. We demonstrate this with a little example.
Let $\catfp$ be the 2-category of finite product categories and $\catcat$ the
2-category of categories.
We define two 2-functors $\funf,\fung:\catfp\to\catcat$, where
\[
 \funf\flc = \flc\times\flc\qtext{and}\fung\flc = \flc
\]
Now we want to define a transformation $\pseta:\funf\to\fung$ by
\[
 \pseta_\flc(C,D) = C\times D.
\]
This only makes sense if we have chosen products, otherwise the object part of
$\pseta_\flc$ is not well defined!
On an informal level, one may be content to define such functors up to
isomorphism, but at the latest when it comes to verifying coherence axioms of
2-categorical constructions, as we will have to do in Section~\ref{sec_tttc},
we really have to be precise what we are talking about.

Therefore, in the following whenever we talk about categories with certain
limits or colimits, we implicitly require chosen such objects. This is
equivalent to equipping the category with limit/colimit functors, because the
morphism parts can
be
inferred by universality.

\section{Pre-equipments}\label{sec_pre-equipments}

This section introduces the categorical backdrop to make sense of our analysis
of the \tttc{}.

The overall aim of the article is to characterize the \tttc{} as a certain
type of biadjunction.
Since biadjunctions are naturally encountered in three dimensional categories
(just like abstract (1-)adjunctions can be defined in arbitrary bicategories), 
we will explain how pre-equipments form such a three dimensional category.
Before we can do this, however, we have to take yet another step back, and
explain the notion of three dimensional category that we are going to use:
Verity's bicategory enriched categories.

Almost everything in this section can be found in Verity's
thesis~\cite{verity_enriched}, but for reasons of self containedness, because
the thesis has not (yet)
 been published, and since the ideas are introduced
there in much greater generality than we need, we repeat the necessary
definitions and constructions here (mostly without proofs).

In the first subsection, we explain the concept of bicategory enriched
category, which is closely related to --- but more general than --- the notion
of Gray-category, and the abstract concept of biadjunction in bicategory
enriched categories.

In the second subsection, we will then introduce \emph{pre-equipments}, explain
how they form a bicategory enriched category, and will have a closer look at
the ensuing notion of biadjunction between pre-equipments.

\subsection{Bicategory enriched categories}

This section is the attempt to summarize the relevant parts of
Section~1.3 of Verity's thesis~\cite{verity_enriched}.

As bicategory enriched categories are related to the more familiar notion
of Gray categories, we begin by
recalling the ideas behind the latter notion. 

Informally, a Gray category is a
three dimensional category whose hom-objects are 2-categories, where
1-cells induce strictly 2-functorial pre- and postcomposition operations $(-)
f$ and $f
(-)$, and which does \emph{not} have
a
primitive `parallel' composition operation $\theta\eta$ for 2-cells
\[
\xymatrix{
\twocata\xtwocell[r]{}^f_{f'}{\eta} & \twocatb\xtwocell[r]{}^g_{g'}{\theta} &
\twocatc
},
\]
but only specified coherent exchange isomorphisms of type
\[
\theta f'\circ g\eta\eiso g'\eta\circ\theta f\qquad 
\cxymatrix{
fg\ar[r]\ar[d] 
\xtwocell[dr]{}<\omit>{\cong}
& f'g\ar[d]\\
fg'\ar[r] & f'g'
}
\]
between sequentializations of the parallel composition.

Formally, Gray categories are defined as categories enriched in the
category of 2-categories
equipped with a certain symmetric monoidal product, the \emph{Gray product}.
The Gray product is characterized by the natural bijection
\begin{equation}\label{eq_bij_gray}
\cattwocat(\twocata\otimes_G
\twocatb,\twocatc)\cong\cattwocat(\twocata,[\twocatb,\twocatc]),
\end{equation}
where $\cattwocat(-,-)$ denotes the set of 2-functors between two
2-categories, and $[-,-]$ denotes the 2-category of 2-functors,
pseudo-transformations and modifications between two 2-categories.

For bicategory enriched categories, the idea is the same, except that we
replace 2-categories by bicategories and 2-functors by pseudofunctors%
\footnote{
Note that the main complication does not arise from the replacement of
2-categories by bicategories, but rather from the replacement of 2-functors by
pseudofunctors. In fact, we could have done all the definitions here using
only 2-categories, as this is all we will need later, but I opted for
bicategories, since it is closer to Verity's presentation, and the additional
effort is negligible.
}. However, we run into problems if we want to adapt the technique for Gray
categories directly, since there is no tensor product on
2-categories/bicategories satisfying an equation like \eqref{eq_bij_gray} if
we replace 2-functors by pseudofunctors in the definition of $[-,-]$. The
solution is to replace enrichment in monoidal categories by enrichment in
multicategories --- it turns out that there exists a multicategory structure on
bicategories which behaves the way we want. The central definition is the
following:
\begin{definition}[$n$-homomorphism]
Let $\twocata_1,\dots,\twocata_n,\twocatb$ be bicategories. An
\emph{$n$-homomorphism}
\[
F:\twocata_1,\dots,\twocata_n\to\twocatb
\]
is given by
\begin{itemize}
\item An object $F(A_1,\dots,A_n)\in\twocatb$ for each $n$-tuple
$(A_1,\dots,A_n)$ of objects with $A_i\in\twocata_i$.
\item For each $1\leq i\leq n$ and each $(n-1)$-tuple
$(A_l)_{l\neq n}$ of objects with $A_i\in\twocata_i$ a
\emph{pseudofunctor}
\[
F(A_1,\dots,A_{i-1},-,A_{i+1},\dots,A_n):\twocata_i\to \twocatb
\]
enriching the mapping on objects. We will often abbreviate this pseudofunctor
by $F(-_i)$ omitting the constant objects.
\item
For all $1\leq i < j \leq n$, all corresponding $(n-2)$-tuples of objects
(suppressed in the notation), and all $f_i:A_i\to A_i',f_j:A_j\to A_j'$
isomorphic 2-cells
\[
\xymatrix@+.4cm{
F(A_i,A_j)\ar[r]^{F(A_i,f_j)}\ar[d]_{F(f_i,A_j)}
\xtwocell[rd]{}<\omit>{\phantom{xxxxx}F(f_i,f_j)} &
F(A_i,A'_j)\ar[d]^{F(f_i,A'_j)}\\
F(A'_i,A_j)\ar[r]_{F(A'_i,f_j)} & F(A'_i,A'_j)
}
\]
such that
\begin{itemize}
\item
The 1-cells $F(f_i)$ together with the 2-cells $F(f_i,f_j)$ give rise to
pseudo-transformations of type
\[
F(A_i,-_j)\to F(A_i',-_j),
\]
\item
The 1-cells $F(f_j)$ together with the 2-cells $F(f_i,f_j)$ give rise to
pseudo-transformations of type
\[
F(-_i,A_j)\to F(-_i,A_j'),
\]
\item
For each triple $1\leq i<j<k\leq n$, for all $f_i:A_i\to A_i',f_j:A_j\to
A_j',f_k:A_k\to
A_k'$ (and for all implicit $(n-3)$-tuples of objects), we have
\[
\vcenter{\xymatrix@R-6mm@C-11mm{
&& FA_iA_jA_k\xtwocell[dd]{}<\omit>{^}\ar[dll]\ar[drr] \\
FA_i'A_jA_k\ar[dd]\xtwocell[dddrr]{}<\omit>{^}\ar[rrd] &&&&
FA_iA_jA_k'\xtwocell[dddll]{}<\omit>{^}\ar[dd]\ar[lld] \\
&& FA_i'A_jA_k'\ar[dd]\\
FA_i'A_j'A_k\ar[drr] &&&& FA_iA_j'A_k'\ar[dll] \\
&& FA_i'A_j'A_k' \\
}}
=
\vcenter{\xymatrix@R-6mm@C-11mm{
&&
FA_iA_jA_k\ar[dll]\ar[drr]\ar[dd]\xtwocell[dddll]{}<\omit>{^}\xtwocell[dddrr]{}
<\omit>{^} \\
FA_i'A_jA_k\ar[dd] &&&&
FA_iA_jA_k'\ar[dd] \\
&& FA_iA_j'A_k\ar[lld]\ar[rrd]\xtwocell[dd]{}<\omit>{^}\\
FA_i'A_j'A_k\ar[drr] &&&& FA_iA_j'A_k'\ar[dll] \\
&& FA_i'A_j'A_k' \\
}}.
\]
\end{itemize}
\end{itemize}
\end{definition}
Observe that a $0$-homomorphism is just an object of $\twocatb$, and a
$1$-homomorphism is a pseudofunctor.

The next step would be the definition of composition of $n$-homomorphisms, and
the verification of the multicategory axioms. We won't give details here, since
there are no surprises. The definition of composition is just `what you would
expect', and
for the verification that the ensuing structure satisfies the axioms of a
symmetric multicategory, we refer to Verity~\cite{verity_enriched}.

The following lemma is in analogy to~\eqref{eq_bij_gray}.
\begin{lemma}
Let $\twocata_1,\dots,\twocata_{n+1},\twocatb$ be bicategories.
There are natural bijections
\[
\nhom(\twocata_1,\dots,\twocata_{n+1};\twocatb)\cong\nhom
(\twocata_1,\dots,\twocata_{n};\sem{\twocata_{n+1},\twocatb}),
\]
where $\nhom$ denotes sets of $n$-homomorphisms, and $\sem{-,-}$ denotes
the bicategory of pseudofunctors, pseudo-transformations and modifications.
\end{lemma}

A bicategory enriched category is now just given by a set $\twocatx_0$ of
objects, for each
pair $X,Y$ of objects a bicategory $\twocatx(X,Y)$, identity $0$-homomorphisms
$\id_X$ (which are just objects of $\twocatx(X,X)$), and composition
$2$-homomorphisms
\[
\mathrm{comp_{X,Y,Z}}:[\twocatx(X,Y),\twocatx(Y,Z)]\longrightarrow\twocatx(X,Z)
,
\]
subject to \emph{strict} associativity and identity axioms.
In bicategory enriched categories, we call the 0-, 1-, and 2-cells of the
bicategories $\twocatx(X,Y)$ \emph{1-, 2-, and 3-cells} of the bicategory
enriched category, respectively, and we denote horizontal composition of
$1$-, $2$- and $3$-cells by juxtaposition (i.e.\
$\mathrm{comp}_{X,Y,Z}(f,g)=gf$), vertical composition of $2$- and
$3$-cells by $( - \circ -)$, and depth-wise composition of $3$-cells by
$(-\cdot-)$. For $\eta:f\to f'$ in $\twocatx(X,Y)$ and $\theta:g\to g'$ in
$\twocatx(Y,Z)$ we denote the exchange isomorphism for horizontal composition
by
\[
\theta\eta: \theta f'\circ g\eta\eiso g'\eta\circ\theta f.
\]
In pasting form this looks like
\[
\vcenter{\xymatrix@R-2mm@C-2mm{
gf\ar[r]^{g\eta}\ar[d]_{\theta f}\xtwocell[rd]{}<\omit>{\theta\eta} &
gf'\ar[d]^{\theta f'}\\
g'f\ar[r]_{g'\eta} & g'f'
}},
\]
and in string diagrams we denote exchange isomorphisms by braidings\footnote{%
The notation as a braiding is motivated by thinking about bicategory enriched
categories in a three dimensional way (the string diagrams that we use and
that live `locally' in a two dimensional section $\twocatx(X,Y)$ of a
bicategory enriched category can actually be viewed as projections of surface
diagrams), but as mentioned earlier we don't want to talk too much about string
diagrams themselves, so for us the notation as braiding is just a definition
of a shorthand for a pasting diagram denoting a 2-cell in $\twocatx(X,Z)$.
}
\[
\sdi{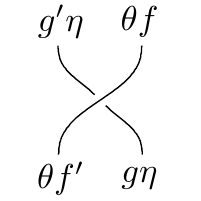}.
\]

Now that we know what a bicategory enriched category is, we can finally
introduce
the desired abstract notion of biadjunction.

\begin{definition}
Let $\twocatx$ be a bicategory enriched category, and let $A,B$ be objects of
$\twocatx$. A biadjunction between $A$ and $B$ is given by
\begin{align*}
&\text{\textbullet{} 1-cells} & f &: A\to B & g&:B\to A,\\
&\text{\textbullet{} 2-cells} & \eta&: \id_A\to gf &\ve &:
fg\to \id_B\\
&\text{\textbullet{} \emph{invertible} 3-cells} & \mu &:  \id_g\eiso
g\ve\circ \eta g & \nu &: \ve f \circ f\eta \eiso \id_f
\end{align*}
such that the diagrams
\[
\vcenter{\xymatrix@+4mm{
& \id_A\ar[dl]_\eta\ar[dr]^\eta\xtwocell[dd(.65)]{}<\omit>{^\eta\eta}\\
gf\ar[r]|{\eta gf}\ar[dr]_\id\xtwocell[dr(1.1)]{}<\omit>{^<-2>\mu f} &
gfgf\ar[d]|{g\ve f} &
\ar[l]|{gf\eta}\ar[dl]^\id gf\xtwocell[dl(1.1)]{}<\omit>{^<2>g\nu}\\
& gf
}}\qquad\quad
\vcenter{\xymatrix@+4mm{
&fg\ar[dl]_\id\ar[dr]^\id\ar[d]|{f\eta g}
\xtwocell[dl(1.1)]{}<\omit>{^<-3>f\mu} \xtwocell[dr(1.1)]{}<\omit>{^<3>\nu g}\\
fg\ar[dr]_\ve & fgfg\ar[l]|{fg\ve}\ar[r]|{\ve
fg}\xtwocell[d]{}<\omit>{^\ve\ve} & fg\ar[dl]^\ve \\
& \id_B
}}
\]
of isomorphic 3-cells compose to identities in $\twocatx(A,A)$ and
$\twocatx(B,B)$, respectively. Note that strictly speaking, these
diagrams are not well typed, as e.g.\ the domain of $\mu f$ is not $\id_{gf}$,
but $\id_gf$, and horizontal composition is only pseudofunctorial. We omit the
constraint isomorphisms since they are easy to fill in, and the diagrams are
clearer and easier to memorize in this form.
\end{definition}
For reference, here are the axioms for biadjoints in string diagrammatic
notation:
\begin{equation}
\sdi{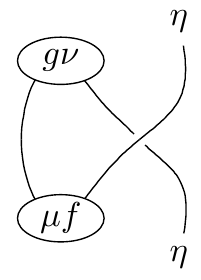}\;\;=\!\!\!\sdi{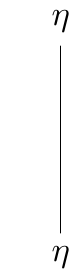}\quad\qtext{and}\quad\sdi{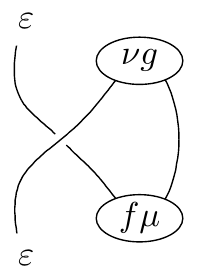}
\;\;=\!\!\!\sdi{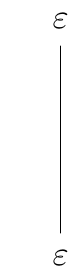}\label{eq_spec_adj_ax}
\end{equation}
Observe that they are rotated and reflected relative to the pasting diagrams
to conform with our convention for the orientation of string diagrams.
Furthermore, as for the pasting diagram version there are some hidden
constraint isomorphisms, since e.g. the 3-cell $g\nu$ has type $g(\ve f\circ
f\eta)\to g\,\id_f$, but its environment in the diagram expects the type $g\ve
f\circ gf\eta\to\id_{gf}$.

We remark that Verity does not require the axioms in his definition of
biadjunction. He calls a biadjunction that additionally satisfies the axioms a
\emph{locally adjoint biadjunction}.

Since we want to use biadjunctions to characterize things, we attach great
value to the following lemma, which is a categorification of the fact that
adjoints are unique up to isomorphism.
\begin{lemma}
Let $\twocatx$ be a bicategory enriched category, and let
\[
(f\dashv g:B\to A,\eta,\ve,\mu,\nu)\qtext{and}(f'\dashv g:B\to
A,\eta',\ve',\mu',\nu')
\]
be two biadjunctions sharing the same right adjoint $g$. Then $f$ and $f'$ are
equivalent.
\end{lemma}
\begin{proof}
The 2-cells between $f$ and $f'$ are given by $\ve f'\circ f\eta':f\to f'$ and
$\ve' f\circ f'\eta:f'\to f$. The fact that they are mutually inverse
equivalences is witnessed by the isomorphic 3-cells
$\alpha:\id_f\eiso\ve' f\circ f'\eta \circ \ve f'\circ f\eta'$
and 
$\beta:\ve f'\circ f\eta' \circ \ve' f\circ f'\eta\eiso\id_f'$
which are defined as
\[
\alpha=\!\!\!\!\!\!\sdi{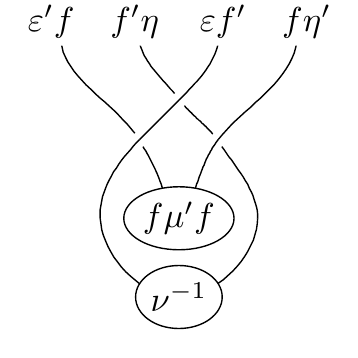}\qquad\quad\beta=\!\!\!\!\!\!\sdi{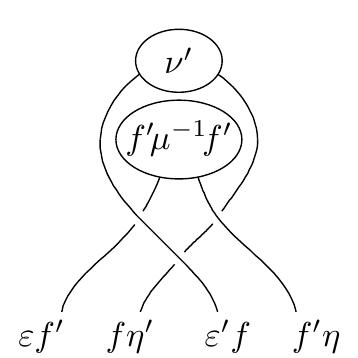}
\]
The interested reader is invited to prove that this equivalence is even an
\emph{adjoint} equivalence.
\end{proof}

\subsection{Pre-equipments}\label{suse_preequs}

We will now introduce the bicategory enriched category $\spec$ of
pre-equipments and special functors and
have a closer look on its biadjunctions, which we call \emph{special}
biadjunctions. This concept is the goal of our
higher dimensional `digressions' --- we will later characterize the \tttc{} as
a special biadjunction between triposes and toposes.

A pre-equipment is almost the same as what Verity calls a \emph{weak proarrow
equipment} (almost, since he doesn't have the closedness condition under
vertical
isomorphisms and furthermore he considers bicategories, not 2-categories), but
the bicategory enriched categories that he considers are
bigger, because his notions of morphisms and transformations (he studies
several of them) are more general than the one we are interested in. We will
elaborate on this after giving our definitions.
\begin{definition}\label{def_preequipment}
\begin{enumerate}
\item\label{def_preequipment_preequipment}
A \emph{pre-equipment} is given by a $2$-category $\twocatc$ together with a
designated subcategory
$\twocatc_r$ of the $1$-cells which is closed under vertical isomorphisms.

We think of the $1$-cells in
$\twocatc_r$ as particularly `nice' arrows and we call them \emph{\regular{}
1-cells}. 

We call a pre-equipment \emph{\tworeg}%
\footnote{The `geometric' refers to \emph{geometric morphism}. We view
adjunctions in pre-equipments as geometric morphisms, and
Lemma~\ref{lem_functors_between_regular_dc} says that these are preserved by
special functors between geometric pre-equipments.}
if all left adjoints in it are \regular.
\item\label{def_preequipment_sfun}
A \emph{special functor} between
pre-equipments $\twocatc$ and $\twocatd$ is an oplax functor
$F:\twocatc\to\twocatd$ such that $Ff$ is a \regular{} 1-cell whenever $f$ is a
\regular{} 1-cell, all identity constraints $F \id_A\to \id_{FA}$ are
invertible,
and the composition constraints $F(gf)\to Fg\,Ff$ are invertible whenever $g$
is a \regular{} 1-cell.

\item\label{def_preequipment_strans}
A \emph{special transformation} between special functors $F,G$ is an
oplax (see Section~\ref{suse_conv}) 
transformation
$\eta:F\to G$ such that all $\eta_A$ are \regular{} 1-cells and the naturality
constraint
$\eta_B\, Ff\to Gf\,\eta_A$ is invertible whenever $f$ is a \regular{} 1-cell.

\end{enumerate}
\end{definition}
Every pre-equipment $\twocatc$ gives rise to a double category
$\widetilde{\twocatc}$ where the vertical 1-cells are the 1-cells of
$\twocatc$, the horizontal arrows are the 1-cells of $\twocatc_r$, and 
\[\qtext{2-cells}
\vcenter{\xymatrix@1@-5mm{A\ar[r]\ar[d] & B\ar[d]\xtwocell[dl]{}<\omit>{^}\\
C\ar[r] & D}} \qtext{in $\widetilde{\twocatc}$ are 2-cells}
\vcenter{\xymatrix@1@-5mm{A\ar[r]\ar[d]\xtwocell[dr]{}<\omit>{^} & B\ar[d]\\
C\ar[r] & D}} \qtext{in $\twocatc$.}\]
 In Verity's bicategory enriched categories
of equipments (see~\cite[Sections~1.4
and~1.5]{verity_enriched}), the 1-cells are certain double functors between
these induced double categories, which are `strong' in horizontal direction and
lax or oplax in vertical direction. Special functors in the sense of the
previous definition give rise to this kind of double functors, but not every
double functor comes from a special functor. This is because a 1-cell in
$\twocatc_r$ appears in $\widetilde{\twocatc}$ as a horizontal and a vertical
cell, but these two need not to be mapped to the same or isomorphic 1-cells by
a
double functor $F:\widetilde{\twocatc}\to\widetilde{\twocatd}$ in the sense of
Verity.

Now, we want to prove that pre-equipments, special functors, special
transformations and modifications form a \bec. If we
wanted to minimize our effort in doing so, we could just prove that the
morphisms and transformations that we consider are special cases of Verity's
\emph{comorphisms} and transformations, and are closed under composition. 

However, to present a more closed flow of ideas, we prefer to describe the
steps which are necessary to establish directly that the given definitions
give rise to a bicategory enriched category. Since the proofs are
for the most part straightforward once you know what to do, we do not prove
every little detail, but only remark on subtleties and important points.

It is well known that oplax functors, oplax transformations, and modifications
between 2-categories $\twocata, \twocatb$ form a 2-category
$\oplax(\twocata,\twocatb)$ (see e.g.\ \cite[Section~2.0]{leinster1998basic}).
Moreover, it is easy to verify that special transformations are closed under
composition, and thus for pre-equipments $\twocatc, \twocatd$, there is a
locally full sub-2-category $\preq(\twocatc,\twocatd)$ of
$\oplax(\twocatc,\twocatd)$ which consists of special functors, special
transformations and modifications.
For pre-equipments $\twocatc,\twocatd,\twocate$, we have to define composition
2-homomorphisms
\[
\comp_{\twocatc,\twocatd,\twocate}:\preq(\twocatc,\twocatd),\preq(\twocatd,
\twocate)\longrightarrow\preq(\twocatc,\twocate).
\]
For special functors $F,G$, $\comp(F,G)$ is just the composition of
oplax functors
(which is again special as is easily seen),
and the definition of the pseudofunctors $\comp(F,-)$ is also straightforward.
Postcomposition $\comp(-,G)$ is more interesting. Crucial here is the
observation that in the world of pseudofunctors and pseudo-transformations,
every $G\in\pseudo(\twocatd,\twocate)$ induces a pseudofunctor
\[
\comp(-,G):\pseudo(\twocatc,\twocatd)\to\pseudo(\twocatc,\twocate),
\]
but this does \emph{not} generalize to oplax functors and transformations. The
reason is that for a pseudo-transformation $\eta:F\to F':\twocatc\to\twocatd$
and $f:C\to C'$, the constraint 2-cell $G\eta_f$ is defined by the pasting
diagram
\begin{equation}\label{eq_noncomposable}
\vcenter{\xymatrix@!@+1.1pc{
GFC
	\ar[r]^{GFf}
	\ar@/^1.2pc/[rd]|<<<<<{G(\eta_{C'}\circ Ff)}^{}="2"
	\ar@/_1.2pc/[rd]|>>>>>{G(F'\!\!f \circ \eta_{C'})}^{}="1"
	\ar[d]_{G\eta_C} &
GFC'
	\ar[d]^{G\eta_{C'}}
\\
GF'C
	\ar[r]_{GF'f} &
GF'C'
\ar@{<=}"2,1"*+\frm{};"1"*+\frm{}^>>>{}
\ar@{<=}"1"*+\frm{};"2"*+\frm{}^{}
\ar@{=>}"2"*+\frm{};"1,2"*+\frm{}_<<<{\!\!\cong}
}},
\end{equation}
but this only makes sense if the upper right composition constraint is
invertible. Now this is not the case in general for oplax functors and
transformations, but it \emph{is} whenever $G$ and $\eta$ are special since
then
$\eta_{C'}$ is regular, and this implies the invertibility of the composition
constraint by the definition of special functor. Postcomposition $G\alpha$ of
modifications with functors is easy again, and in the end, $\comp(-,G)$ is a
pseudofunctor for the same reasons as it is in the pseudo case.

To make $\comp(-,-)$ into a 2-homomorphism, we still have to define
modifications $\theta\eta:G'\eta\circ\theta F\to\theta F'\circ G\eta$ for
special transformations $\eta:F\to F'$ and $\theta:G\to G'$ and check that they
have the desired properties. There is only one way to do define $\theta\eta$ ---
given $C\in\twocatc$, we define the component of $\theta\eta$ at $C$ as the
constraint cell of $\theta$ at $\eta_C$, i.e.\ $(\theta\eta)_C =
\theta_{\eta_C}$. It follows from the fact that $\eta$ and $\theta$ are special
that this is an isomorphism, and we leave the remaining verification that the
such defined 2-cells give rise to pseudo-natural transformations
\[
 \comp(\eta,-):\comp(F,-)\to\comp(F',-)
\]
and
\[\comp(-,\theta):\comp(-,
G)\to\comp(-,G')
\]
to the reader.

The identities of our bicategory enriched category are just given by
identity-2-functors $\id_\twocatc\in\preq(\twocatc,\twocatc)$, and the
verifications of associativity and identity axioms do not bear any surprises
either. We are thus able to state:
\begin{lemma}
Pre-equipments, special functors, special transformations and modifications
together form a bicategory enriched category $\preq$.
\qed
\end{lemma}

As a first example, we define the pre-equipment of toposes.  In
Sections~\ref{sec_triposes} and~\ref{sec_q_toposes}, we will furthermore
introduce the pre-equipments $\cattrip$ and $\catqtop$ of triposes and
q-toposes, respectively.

\begin{example}\label{ex_pre-eq_toposes}
The pre-equipment $\cattop$ of toposes has the 2-category of toposes, finite
limit
preserving functors and arbitrary natural transformations as underlying
2-category, and regular (i.e.\ epi preserving) functors as \regular{} 1-cells.
$\cattop$ is a \tworeg{} pre-equipment, since epimorphisms are preserved by
left
adjoints.
\end{example}

We call a biadjunction in $\spec$ a \emph{special biadjunction} (we do not use
the term biadjunction of pre-equipments since this expression should be
reserved for the more general double categorical notion). Special
biadjunctions enjoy the following interesting property.
\begin{lemma}\label{lem_right_strong}
Let $\funf\adj \funu :\twocatd\to\twocatc$ be a special biadjunction.
Then $\funu$ is a strong functor.
\qed
\end{lemma}
A proof of a more general lemma (with double functors instead of special
functors) appears in~\cite{verity_enriched}. %
Similar results are also known for monoidal categories and double categories,
the first systematic treatment of this phenomenon (for lax morphisms of
pseudo-algebras) is \cite[Theorem~1.5]{kelly1974doctrinal}. We do not prove the
lemma here, because we don't really need it --- for the biadjunctions that we
consider we know that the right adjoints are pseudofunctors anyway.
But we chose to mention the statement, since it somehow fits into the
picture: The right adjoints in our case are forgetful functors --- there is no
reason for them to have non-invertible composition constraints. The left
adjoints, however, are free constructions which naturally have more `degrees of
freedom', which in a sense justifies them being oplax.

\medskip

Finally, some remarks about \emph{geometric} pre-equipments. Wood's
\emph{proarrow equipments} have the property that all 1-cells of the
designated subcategory have right adjoints, and this is the important property
for abstract category theory, for which proarrow equipments were introduced,
since it allows to give an abstract treatment of phenomena related to
contravariance.

In the pre-equipments that we consider, the reverse inclusion holds,
i.e.\ every left adjoint is in the designated subcategory (for example inverse
image parts of geometric morphisms are regular), and we call a pre-equipment
having this property \emph{geometric}. The following lemma is an easy
observation about special functors between geometric pre-equipments.

\begin{lemma}\label{lem_functors_between_regular_dc}
Let $\twocatc$ be a geometric pre-equipment, and let $F:\twocatc\to\twocatd$
be
a special functor. If $(f\dashv u:B\to A,\eta,\ve)$ is an adjunction in
$\twocatc$, then $(Ff\dashv Fu:FB\to FA,\phi_{u,f}\;
F\!\eta\;\phi_{\id_A}^{-1},\phi_{\id_B}\;F\ve\;\phi_{f,u}^{-1})$ is an
adjunction in $\twocatd$ (By $\phi$ we denote the identity and composition
constraints of $F$).

If $f\dashv u$ is a reflection (i.e., has isomorphic counit), then so is
$Ff\dashv Fu$.
\qed
\end{lemma}

\subsection{Equipments and related concepts in the literature}

Proarrow equipments were first introduced by Wood in
\cite{wood_proarrows_i} as a framework for abstract category theory. 

Lax functors between double categories were considered, besides by Verity, in
the work
of Grandis and Par\'e~\cite{grandis1999limits}. Their definitions are a bit
different since they use pseudo-double categories instead of double
bicategories, as Verity does. Shulman gives yet another --- a bit more
restrictive and therefore easier and shorter --- variant of the definitions in
his work on framed bicategories~\cite{shulman2008framed}. Shulman's
bicategories are `pseudocategories' in $\catcat$ (just like monoidal
categories are pseudomonoids in $\catcat$), thus they have strict composition
vertically and bicategorical composition horizontally. Furthermore, his
lax double functors commute with vertical composition on the nose.

The notion of special transformation appeared (without name) already in 1993 in
Johnstone's~\cite[Lemma~1.1]{johnstone93} where it is used to define a less
general version of what we call special biadjunction, under
the name
\emph{semi-oplax adjunction}. This article is notable since it
was a starting point for the present work.

We adopted the adjective \emph{special} from~\cite{dms03}, where Day,
McCrudden and Street define special
functors for those pre-equipments in which the regular arrows
coincide with the left adjoints.

\subsubsection*{Vertical and horizontal}

When associating a double category to a pre-equipment, we have a choice to
make, namely whether we want to view the \regular{} 1-cells as horizontal or
vertical 1-cells in the double category. In the context of general
double-categories, this correpsonds to the question whether we want to view
lax double functors as strong on the horizontal or on the vertical 2-category.

The first convention, where regular 1-cells are horizontal and lax double
functors are horizontally strong, is used by Grandis, Par\'e and
Verity, whereas for Shulman, regular 1-cells are horizontal, and lax double
functors are vertically strong (and even strict).

We do not use double
categories explicitly, but the fact
that we usually draw the components of natural transformations vertically in
naturality squares corresponds the convention used by Shulman (since
components of special transformations are regular).

\section{Triposes}\label{sec_triposes}

In one sentence, triposes are fibrational models of non-extensional
intuitionistic higher order logic. For a general introduction to
fibrations in categorical logic and their internal language, we refer to
Jacobs' book~\cite{jacobs2001categorical}. We will now give
the definition of tripos; what it means for a tripos to be a model of higher
order logic will be explained in Section~\ref{suse:ilatri}.
\begin{definition}
A \emph{Heyting algebra} is a partial order that is bicartesian closed as a
category. More explicitly, it is a poset with finite meets, finite joins
and an operation $(-
\imp -)$ universally characterized by
\[
\varphi\wedge\psi\leq\gamma\quad\text{iff}\quad \varphi\leq\psi\imp\gamma.
\]
\end{definition}

\begin{definition}\label{def_tripos}
Let $\flc$ be a category with finite products. A \emph{tripos over $\flc$} is a
fibration
\[
\trip:\flx\to\flc
\]
such that
\begin{enumerate}
 \item All fibers of $\trip$ are Heyting algebras.
 \item Reindexing along morphisms in $\flc$ preserves all structure of Heyting
algebras.
\item For every
$f:A\to B$ in $\flc$, the reindexing map $f^*:\trip_B\to\trip_A$
has left and right adjoints
\[ \exists_f\dashv f^*\dashv\forall_f,\]
such that for every pair $f:A\to B,\quad g:X\to Y$ of morphisms in $\flc$ and
all $\varphi\in\trip_{B\times X}$, we have  
\begin{equation}\label{eq_bc}
Q_{A\times g}\left((f\times X)^*\varphi\right) = (f\times Y)^*(Q_{B\times
g}\varphi),
\end{equation}
where $Q$ is either $\forall$ or $\exists$.
\item\label{def_tripos_chi} $\trip$ has \emph{weak power objects}, i.e., for
every
$A\in\objects{\flc}$ there is  an object $\tripower A\in\flc$ and a predicate
$(\ni_A)\; \in\trip_{\tripower A\times A}$ such that for all predicates
$\varphi\in\trip_{C\times A}$ we are given a map\footnote{This map is not
supposed to be uniquely determined by the stated property, but we assume that
the tripos is equipped with a choice of such maps.}
$\chi_A(\varphi):C\to\tripower A$ such that $\varphi =(\chi_A(\varphi)\times
A)^*(\ni_A)$, which is written diagramatically%
\footnote{It has to be
explained how to read diagram~\eqref{eq_dia_tripos}. Here, we are
using a notation that is very common for fibrations --- by drawing one object
above another, e.g.\ $\varphi$ over $C\times A$, we assert that $\varphi$ is
in the fiber over $C\times A$, i.e.\ $\trip(\varphi) = C \times A$, and in
the
same way for morphisms. We use wavy arrows $\xymatrix{\ar@{~>}[r]&}$ to
denote cartesian morphisms. Thus, the diagram says that
$\varphi$ is the cartesian lifting of $\ni_A$ along $\chi_A(\varphi)\times
A$.
}
as
\begin{equation}\label{eq_dia_tripos}
 \vcenter{\xymatrix@C+2pc{
\varphi \ar@{~>}[r] & \ni_A \\
C\times A \ar[r]_<<<<<<<<{\chi_A(\varphi)\times A} & \tripower(A) \times A
}}.
\end{equation}
\end{enumerate}
\end{definition}

\begin{remark}
The third clause of Definition~\ref{def_tripos} requires some clarifications.
Condition~\eqref{eq_bc} is the \emph{Beck-Chevalley condition}, we will
abbreviate it by (BC). It
is usually stated in the form
\begin{quotation}
\noindent
``For every pullback square
$\vcenter{\xymatrix@-17pt{P \ar[r]_q\ar[d]_p& B\ar[d]^g \\ A\ar[r]^f & C}}$
in
$\flc$ and all $\varphi\in \trip_A$, we have $Q_q(p^*\varphi)\cong
g^*(Q_f(\varphi))$ where $Q$ is either of $\forall, \exists$.''
\end{quotation}
This definition is not appropriate in our setting, since we
only assume $\flc$ to have products (not arbitary finite limits), and in
Section~\ref{suse:trimorphs} we consider functors between the bases of
triposes which only preserve finite products. The deep reason why we have to
abandon general finite limits in favor of finite products will become apparent
in Section~\ref{sec_f_s_unit} --- the functors $D_\trip$ defined there only
preserve products.

It has long been observed that the full strength of the classical
Beck-Chevalley condition is not necessary to ensure soundness of the
interpration of logical systems. Jacobs~\cite{jacobs2001categorical} gives
definitions using the pullback squares
\begin{equation}\label{eq_bc_jacobs}
\vcenter{\xymatrix{
A\times X\ar[r]_-{A\times !}\ar[d]_{f\times X} & A\ar[d]^f\\
B\times X \ar[r]^-{B\times !} & B
}}
\qtext{and}
\vcenter{\xymatrix{A\times
X\ar[r]_-{A\times \delta}\ar[d]_{f\times X} & A \times
X\times X\ar[d]^{f\times X\times X}\\
B\times X \ar[r]^-{B\times \delta} & B\times X\times X}}
\end{equation}
which are definable from finite product structure \cite[Definitions~1.9.1 and
3.4.1]{jacobs2001categorical}.
\eqref{eq_bc} corresponds to the classical condition restricted to squares of
the form
\begin{equation}\label{eq_bc_square}
\vcenter{\xymatrix{A\times X\ar[r]_{A\times g}\ar[d]_{f\times X} & A \times
Y\ar[d]^{f\times Y}\\
B\times X \ar[r]^{B\times g} & B\times Y}}.
\end{equation}
The class of squares of this form encompasses all squares used by Jacobs, in a
slightly more concise way. Moreover, it expresses precisely the desired
property, namely the commutation of substitution and (generalized)
quantification.

As a side note, Jacobs' version of (BC) and \eqref{eq_bc} are equivalent,
i.e., any fibration of Heyting algebras with $\exists$ satisfying (BC) for
squares of the forms~\eqref{eq_bc_jacobs} already satisfies (BC) for all
squares of the form~\eqref{eq_bc_square}. This can be seen by analyzing the
proofs of the substitution lemma~\ref{lem_subs_trip} and the soundness
theorem~\ref{theo_sound_trip} below. They only require (BC) for the
squares~\eqref{eq_bc_jacobs} as hypothesis, but using the internal logic we
can prove (BC) for all squares of the form~\eqref{eq_bc_square}. In fact, in
our setting even Jacobs' set of squares is redundant --- 
the condition for the
right square in~\eqref{eq_bc_jacobs} can be derived using the equivalence
$(\exists_\delta\varphi)(x,y)\dashv\vdash\varphi(x)\wedge x=y$.
This equivalence is a consequence of the Frobenius law, which in turn follows
from the existence of implication.

An advantage of the phrasing~\eqref{eq_bc} of (BC) is that it does not rely on
projections and diagonals, and thus is still meaningful in a monoidal setting.
Indeed, Shulman proves the monoidal version of the condition for certain
monoidal fibrations in~\cite[Corollary~16.4]{shulman2008framed}.
\end{remark}

\subsection{Interpreting higher order logic in
triposes}\label{suse_hol_in_triposes}\label{suse:ilatri}

In this section, we explain how to interpret languages of higher order logic in
triposes. This provides the basis for the \emph{internal language} of a tripos,
to be presented in the next section. Jacobs'
book~\cite{jacobs2001categorical} gives a careful exposition of how to
interpret different systems of predicate logic in fibrations, but for reasons
of self-containedness, and because the internal language of a tripos will be a
central tool in the following, we give a detailed and explicit description of
the system that we use, how it can be interpreted in a tripos. Then, in the
next
section, we explain how the internal language --- which is the language that
we get for the `maximal' choice of signature --- can be used to reason and
calculate in a tripos.

\begin{definition}
A \emph{signature} for a language of many sorted higher logic is given by a
triple $\Sigma=(\sigs,\sigf,\sigr)$ where 
\begin{itemize}
 \item $\sigs$ is a set of \emph{base types},
\item $\hotsig$ is the set of \emph{higher order types} generated by $\sigs$,
that is the smallest set that contains all elements of $\sigs$ and is closed
under the inductive clauses%
\begin{itemize}
\item
$1\in\hotsig$
\item
$A,B\in\hotsig\quad\imp\quad A\times B\in\hotsig$
\item
$A\in\hotsig\quad\imp\quad \ptype(A)\in\hotsig$,
\end{itemize}
\item $\sigf=(\sigf_{\Delta,A}\;;\;\Delta\in\hotsig^*,A\in\hotsig)$ is a family
of sets of \emph{function symbols}
($\hotsig^*$ is the set of lists of higher order types), where for
$(A_1,\dots,A_n)\equiv \Delta\in\hotsig^*$ and $A\in\hotsig$ we view an
$f\in\sigf_{\Delta,A}$ as a function of type $f:A_1\times\dots\times A_n\to A$.
\item
$\sigr=(\sigr_{\Delta}\;;\;\Delta\in\hotsig^*)$ is a family of sets of
\emph{relation symbols}, where for $\Delta\in\hotsig^*$, we view
$R\in\sigr_\Delta$ as a relation of arity $\Delta$.
\end{itemize}
\end{definition}
From a signature $\Sigma$, we can inductively generate \emph{terms} and
\emph{formulas}. To be able to define the semantics later, we have to keep
track of free variables explicitly, using \emph{contexts}. A context is a list
$\Delta\equiv x_1:A_1,\dots,x_n:A_n$ of variable declarations, where
$A_i\in\hotsig$. We will write terms and formulas in context as $(\Delta\csep
t\ttp B)$ and $(\Delta\csep \varphi)$ (the symbol $\vdash$ is reserved for the
entailment relation between formulas).
Table~\ref{table_hol}  gives the inductive clauses for terms and formulas.
\begin{table}
\noindent\framebox[\textwidth]{\parbox{.96\textwidth}{
\medskip
\textbf{Terms:}
\begin{center}
\begin{tabular}{cc}
\unary{\phantom{|}}{}{\Delta\csep x_i:A_i} &
\unary{\Delta\csep t_i:B_i\quad (1\leq i \leq
n)}{\scriptsize$f\in\sigf_{(B_1,\dots,B_n),C}$}{\Delta\csep
f(t_1,\dots,t_n):C}\\[5mm]
\unary{\Delta\csep t:B_1\times B_2}{$i=1,2$}{\Delta\csep \pi_i(t):B_i} &
\end{tabular}
\end{center}
\textbf{Formulas:}
\begin{center}
\begin{tabular}{c@{$\quad$}c@{$\quad$}c}
\unary{\phantom{\csep}}{}{\Delta\csep\top} &
\unary{\phantom{\csep}}{}{\Delta\csep\bot} & 
\binary{\Delta\csep\varphi}{\Delta\csep\psi}{}{\Delta\csep\varphi\wedge\psi}
\\[.5cm]
\binary{\Delta\csep\varphi}{\Delta\csep\psi}{}{\Delta\csep\varphi\vee\psi} &
\binary{\Delta\csep\varphi}{\Delta\csep\psi}{}{\Delta\csep\varphi\imp\psi} & 
\binary{\Delta\csep s: B}{\Delta\csep t: B}{}{\Delta\csep s = t} \\[.5cm]
\unary{\Delta,y\vtp B\csep \varphi[y]}{}{\Delta\csep \exists y\vtp B\qdot
\varphi[y]}  &
\unary{\Delta,y\vtp B\csep \varphi[y]}{}{\Delta\csep \forall y\vtp B\qdot
\varphi[y]}&
\binary{\Delta\csep s:B}{\Delta\csep t:\ptype B}{}{\Delta\csep s\in t}
\end{tabular}
\[
\unary{\Delta\csep t_i:B_i\quad (i=1,\dots,n)}{
\vbox{\scriptsize\hbox{$R\in\sigr$}\scriptsize\hbox{$a_\sigr(R)=(B_1,\dots,
B_n)$}}}{\Delta\csep R(t_1,\dots,t_n)}
\]

\end{center}
$\Delta\equiv x_1\vtp A_1,\ldots ,x_n\vtp A_n$ denotes a context of typed
variables.
}}
\caption{Terms and formulas in context over a signature
$\Sigma=(\sigs,\sigf,\sigr)$}
\label{table_hol}
\end{table}
For reasons of conciseness and better readability, we will often supress
contexts from the notation for terms and formulas, and types from the notation
for formulas.
We call the collections of terms and of formulas generated from $\Sigma$
together \emph{the language generated by $\Sigma$}, and denote it by
$\holt{\Sigma}$.

\begin{definition}[Interpretation]
Given a signature $\Sigma=(\sigs,\sigf,\sigr)$ and a tripos
$\trip:\flx\to\flc$, we may define an \emph{interpretation} of $\holt{\Sigma}$
in
$\trip$. This works as follows.
\begin{itemize}
\item To each base type symbol $S\in\sigs$, we associate an object
$\sem{S}\in\flc$ in the base of the tripos.
\item We inductively extend this assignment to higher order types using the
rules $\sem{1}=1$, $\sem{A\times B} = \sem{A}\times \sem{B}$ and $\sem{\ptype
A}=\tripower\sem{A}$.
\item The interpretation of a list of types is given by
\[
\sem{A_1,\dots,A_n}=\sem{A_1}\times\dots\times\sem{A_n}
\]
The interpretation of a context $\Delta$ is the interpretation of the
associated list of types (obtained by syntactically removing the variables).
\item To each function symbol $f\in\sigf_{(A_1,\dots,A_n),B}$ we associate a
morphism $\sem{f}:\sem{A_1,\dots,A_n}\to\sem{B}$ in $\flc$.
\item To each relation symbol $R\in\sigr_{(A_1,\dots,A_n)}$ we associate a
predicate $\sem{R}\in\trip_{\sem{A_1,\dots,A_n}}$.
\item Now we can inductively define the semantics of terms by
\begin{align*}
\sem{\Delta\csep x_i:A_i} &= \pi_i \\
\sem{\Delta\csep f(t_1,\dots,t_n):B}&=\sem{f}\circ\langle\sem{\Delta\csep t_1
:A_1},\dots,\sem{\Delta\csep t_n :A_n}\rangle\\
\intertext{and of formulas by}
\sem{\Delta\csep\top} & = \top\in\trip_\sem{\Delta}\\
\sem{\Delta\csep\bot} & = \bot\in\trip_\sem{\Delta}\\
\sem{\Delta\csep\varphi\wedge\psi} &=
\sem{\Delta\csep\varphi}\wedge\sem{\Delta\csep\psi}\\
\sem{\Delta\csep\varphi\vee\psi} &=
\sem{\Delta\csep\varphi}\vee\sem{\Delta\csep\psi}\\
\sem{\Delta\csep\varphi\imp\psi} &=
\sem{\Delta\csep\varphi}\imp\sem{\Delta\csep\psi}\\
\sem{\Delta\csep s=t} &= \langle\sem{\Delta\csep s:B},\sem{\Delta\csep
t:B}\rangle^*(\eq_\sem{B})\\
\sem{\Delta\csep \exists y\vtp B\qdot
\varphi[y]}&=\exists_{\pi^-}\sem{\Delta,x\vtp B\csep \varphi[x]}\\
\sem{\Delta\csep \forall y\vtp B\qdot
\varphi[y]}&=\forall_{\pi^-}\sem{\Delta,x\vtp B\csep\varphi[x]}\\
\sem{\Delta\csep s\in t} &= \langle\sem{\Delta\csep t:\ptype
B},\sem{\Delta\csep s:B}\rangle^*(\ni_\sem{B})\\
\sem{\Delta\csep R(t_1,\dots,t_n)} & = \langle\sem{\Delta\csep t_1:
B_1},\dots,\sem{\Delta\csep t_n:B_n}\rangle^*(\sem{R})\\
\end{align*}
In the line for equality we use the notation $\eq_A=\exists_{\delta_A}(\top)$
where $\delta_A:A\to A\times A$ is the diagonal.
In the clauses for existential and universal quantification, $\pi^-$ denotes
the projection of type $\sem{\Delta,B}\to\sem{\Delta}$.

Observe that the interpretation of terms and formulas is compatible with types,
i.e.,
\[
\sem{\Delta  \csep t : A} : \sem{\Delta}\to\sem{A}\qtext{and}
\sem{\Delta\csep\varphi}\in\trip_{\sem{\Delta}}.
\]
\end{itemize}
\end{definition}
For the remainder of this section, $\Sigma=(\sigs,\sigf,\sigr)$ is a fixed
signature with a fixed interpretation $\sem{-}$ in a tripos
$\trip:\flx\to\fle$. Terms and formulas will always be terms and formulas
generated from $\Sigma$.
\begin{convention}
If $(\Delta\csep\psi)$ is a formula such that $\sem{\Delta\csep\psi}=\top$,
then we say that $\psi$ \emph{holds in $\trip$}. More generally, if
$\sem{\varphi_1}\wedge\dots\wedge\sem{\varphi_n}\leq \sem{\psi}$ holds for
formulas $\varphi_1,\dots,\varphi_n, \psi$ in context $\Delta$, then we say
that the judgment%
\footnote{As for terms and formulas, we will often suppress the context from
the
notation for judgments.}
\[
\Delta\csep\varphi_1,\dots,\varphi_n\vdash\psi
\]
 holds in $\trip$.
\end{convention}

The most important properties of $\sem{-}$ are the \emph{substitution lemma}
and the \emph{soundness theorem}, stated now.
\begin{lemma}[Substitution lemma for triposes]\label{lem_subs_trip}
Let $(\Delta\csep s_i:B_i), 1\leq i\leq n$ and $(\Delta'\csep
t[y_1,\dots,y_n]:C)$ be terms and let $(\Delta'\csep \varphi[y_1,\dots,y_n])$
be a formula, where $\Delta' = y_1\vtp B_1,\dots, y_n\vtp B_n$. Then we have
\begin{enumerate}
\item $\sem{\Delta\csep t[s_1,\dots,s_n]} = \sem{\Delta'\csep
t[y_1,\dots,y_n]}\circ\langle\sem{\Delta\csep s_1},\dots,\sem{\Delta\csep
s_n}\rangle$
\item $\sem{\Delta\csep \varphi[s_1,\dots,s_n]} = \langle\sem{\Delta\csep s_1}
,\dots,\sem{\Delta\csep s_n}\rangle^*(\sem{\Delta'\csep
\varphi[y_1,\dots,y_n]})$
\end{enumerate}
\qed
\end{lemma}
The soundness theorem says that the interpretation that we described is
compatible with derivability of judgments in some logical system, thus we have
to clarify which logical system we use before stating it.
\begin{definition}
\textbf{Non-extensional higher order intuitionistic logic} is intuitionistic
predicate logic with explicit contexts of variables, formalized e.g.\ in
natural deduction, with the additional axiom
\[
\Delta,x\vtp A,y\vtp B\csep\Gamma\vdash \exists ! z\vtp
\, A\!\times\! B\qdot \pi_1(z)=x\wedge\pi_2(z)=y
\]
for product types, and
the additional \emph{comprehension scheme}
\[
\Delta\csep\Gamma\vdash \exists m\vtp\ptype B\; \forall y\in B\qdot y\in
m\leftrightarrow \varphi[y]
\]
for power types, where $(\Delta,y\in B\csep\varphi[y])$ is an arbitrary
formula.

Since the explicit handling of variable contexts $\Delta$ is not contained in
the standard presentation of intuitionistic logic, we give a complete natural
deduction system in Appendix~\ref{sec_nehoil}.
\end{definition}
Now the soundness theorem can be phrased as follows.
\begin{theorem}[Soundness theorem for triposes]\label{theo_sound_trip}
If the judgment \[\Delta\csep \varphi_1,\dots,\varphi_n\vdash \psi\] is
derivable in non-extensional higher order intuitionistic logic, then it holds
in $\trip$.
\qed
\end{theorem}
The substitution lemma and the soundness theorem are proved by induction on the
structure of formulas/terms, and proofs respectively. This is fairly standard
and straightforward, similar proofs can be found
in~\cite{jacobs2001categorical}.

A direct consequence of the soundness theorem is that if we have a theory over
$\Sigma$ generated by a given set of axioms such that all the axioms hold in
$\trip$, then any statement that can be derived from the axioms does also hold
in $\trip$.

\subsection{The internal language of a tripos}

\begin{definition}[The internal language of a tripos]
Given a tripos $\trip:\flx\to\flc$, we define the signature 
$\Sigma^\trip=(\sigs^\trip,\sigf^\trip,\sigr^\trip)$, and at the same time an
interpretation $\semp{-}$ of the language $\holt{\trip}$ generated by
$\Sigma^\trip$ in
$\trip$, as follows:
\begin{itemize}
\item The set of base types is defined as 
\[\sigs_\trip=\objects{\flc} \quad\text{(the set of objects of $\flc$)},\]
 and $\semp{C}=C$ for $C\in\sigs^\trip$.
\item For $\Delta\in\hotsigp^*,A\in\hotsigp$, we define
\[
 \sigf^\trip_{\Delta,A} = \flc(\semp{\Delta},\semp{A})
\]
and $\semp{f}=f$ for $f\in\sigf^\trip_{\Delta,A}$.
\item For $\Delta\in\hotsigp^*$, we define
\[
 \sigr^\trip_{\Delta} = \trip_{\semp{\Delta}}
\]
and $\semp{R}=R$ for $R\in\sigr^\trip_{\Delta}$.
\end{itemize}
The \emph{internal language of} $\trip$ is the language that is generated by
$\Sigma^\trip$.
\end{definition}
In the following, we will use the internal language freely and heavily when
reasoning about triposes. 

The power object of $1$ has a special status since it is the type of
propositions, therefore we introduce the notations
\begin{align*}
\prop&\defpred\tripower 1\\
\triptr(p)&\defpred \exists x\vtp 1\qdot x\in p
\end{align*}
for the power object of $1$ and its element predicate.

\subsection{Tripos morphisms}\label{suse:trimorphs}

Fibrations form 2-categories in a natural way, the 1- and 2-cells being the
fibered functors
and fibered natural transformations (see~\cite[Definition~2.3]{streicherfib}).
For triposes, we only consider fibered
functors that are
compatible with a part of
the logical structure.

\begin{definition}\label{def:tripos-morphism}
 Let $\trip:\flx\to\flc$ and $\triq:\fly\to\fld$ be two triposes.
\begin{itemize}
 \item A \emph{tripos morphism} is a pair of functors $(F,\Phi)$ with
$F:\flc\to\fld$
  and $\Phi:\flx\to\fly$ with the following four properties.
\begin{enumerate}
\item The square
\[
\xymatrix{
\flx 
    \ar[r]^\Phi
    \ar[d]_\trip
    &
\fly
    \ar[d]^\triq
    \\
\flc 
    \ar[r]_F
    &
\fld
}
\]
commutes (on the nose).
\item $\Phi$ maps cartesian arrows to cartesian arrows.
\item $F$ preserves finite products%
.
\item For each $A\in\objects{\flc}$, the restricted functor
$\Phi_A:\trip_A\to\triq_{FA}$ preserves
finite meets.
\end{enumerate}
\item A tripos morphism $(F,\Phi)$ is called \emph{regular} if it satisfies the
following additional condition.
\begin{itemize}
 \item[5.]
$\Phi$ maps cocartesian
arrows to cocartesian arrows.
\end{itemize}
\end{itemize}
\end{definition}
Conditions 1. and 2. in the definition of tripos morphism say that $\Phi$ is a
fibered
functor over $F$. The others are compatibility postulates. Their effect is
best understood in terms
of the internal language. 

In order to express the interaction of tripos
morphisms and the internal language, we need some more definitions.
\begin{definitions}\label{def_coherence-stuff}
Let $\effi:\trip\to\triq$ be a tripos morphism between triposes
$\trip:\flx\to\flc$ and $\triq:\fly\to\fld$.
\begin{enumerate}
\item
We denote by $\folt{\trip}$ the fragment of the internal language of $\trip$
which is first order, i.e.\ without power types and the element predicate, but
with product types, and we denote by $\fots(\trip)$ the corresponding set of
(first order) types.
\item 
For $A\in\fots(\trip)$, we denote by $\into{F}{A}\in\fots(\triq)$ the
type that is obtained by replacing all the occurring base types $C\in{\flc}$ in
$A$ by $FC$.
For a list $\Delta=(C_1,\dots,C_n)$, of objects of $\flc$, we write
$\into{F}{\Delta}$ for the list $(\into{F}{C_1},\dots,\into{F}{C_n})$,
in the same way for contexts.
\item
Since first order types are built up only from finite products which are
preserved by $F$, there are obvious commutation isomorphisms which we name as
follows:
\begin{align*}
\sigma_A&:\sem{\into{F}{A}}\eiso F\semp{A}\\
\sigma_\Delta&:\sem{\into{F}{\Delta}}\eiso F\semp{\Delta}.
\end{align*}
\item 
For a function symbol $f\in\sigf^\trip_{\Delta,A}$, define $F_{\Delta,A}(f)\in
\sigf^\triq_{\into{F}{\Delta},\into{A}{\Delta}}$ as
\[
F_{\Delta,A}(f)=\sigma^{-1}_A \circ Ff \circ
\sigma_\Delta.
\]
\item\label{def_coherence-stuff_relsym}
For a relation symbol $\varphi\in\sigr^\trip_\Delta$, define
$\Phi_\Delta(\varphi)\in\sigr^\triq_{\into{F}{\Delta}}$ by
\[\Phi_\Delta(\varphi)=\sigma_\Delta^*(\Phi\varphi).\]
\item For a term $(\Delta\csep t:A)$ of $\folt{\trip}$, we
denote by
$(\into{F}{\Delta}\csep \into{F}{t}:\into{F}{A})$ the term of $\folt{\triq}$
that is obtained by
replacing each of the occurring function symbols $g\in\sigf^\trip_{\Delta',B}$
in $t$ by $F_{\Delta',B}(g)$. 
\item For a formula $(\Delta\csep \varphi)$ of $\folt{\trip}$, we denote by
$(\into{F}{\Delta}\csep \into{F}{\varphi})$ the formula  of $\folt{\triq}$ that
is obtained by
replacing each of the occurring function symbols $g\in\sigf^\trip_{\Delta',B}$
in $t$ by $F_{\Delta',B}(g)$, and each relation symbol
$\theta\in\sigr^\trip_\Delta$ by $F_\Delta(\theta)$.
\end{enumerate}
\end{definitions}

The interaction of tripos morphisms and internal language is now expressed by
the following lemma.
\begin{lemma}
Let $\effi:\trip\to\triq$ be a tripos morphism.
\begin{itemize}
\item Let $(\Delta\csep t:A)$ be a term in $\folt{\trip}$. We have
\[
\sigma_A\circ \sem{\into{F}{t}} = {F}{\sem{t}}\circ \sigma_\Delta
\]
\item Let $(\Delta\csep \varphi)$ be a formula in $\folt{\trip}$. If $\varphi$
is built from atomic formulas (excluding equality) using only $\wedge$ and
$\top$, then we have
\[\sem{\into{\Phi}{\varphi}} = \sigma_\Delta^*(\Phi\sem{\varphi}).\]
If $\effi$ is regular, then $\varphi$ may also contain $\exists$ and $=$.
\end{itemize}
\qed
\end{lemma}
This lemma looks very much like a categorical coherence theorem, but it is
much easier to prove since we can do it by induction on the structure of terms
and formula.

An important consequence of the lemma is that validity of sequents is
preserved by tripos morphisms. Because we will heavily make use of this fact
later, we spell it out explicitly.
\begin{corollary}\label{cor_trimo-judgement}
Let $\effi:\trip\to\triq$ be a tripos morphism, and let
${\Delta\csep\Gamma\vdash\varphi}$ be a valid judgment in
$\folt{\trip}$
containing no equality and only $\wedge,\top$ as logical symbols. Then
${\into{F}{\Delta}\csep\into{\Phi}{\Gamma}\vdash\into{\Phi}{\varphi}}
$ is valid
in $\triq$. If $\effi$ is regular, then the statement also holds for judgments
containing  $\exists$ and $=$. \qed
\end{corollary}

The next definition gives the 2-cells in the 2-category of triposes.
\begin{definition}
Let $\trip:\flx\to\flc$ and $\triq:\fly\to\fld$ be two triposes and
consider two tripos morphisms $(F,\Phi),(G,\Gamma):\trip\to\triq$. A
\emph{transformation}
from $(F,\Phi)$ to $(G,\Gamma)$ is a natural transformation $\eta:F\imp G$ with
the property
that for all $A\in\objects{\flc}$ and all $\psi\in\trip_A$, we have
\[
a\vtp FA\csep\Phi\psi(a)\vdash\Gamma\psi(\eta_A(a)),
\]
or diagrammatically
\[
\xymatrix{
\psi
    &&
\Phi\psi
    \ar[r]
    &
\Gamma\psi
    \\
A
    &&
FA
    \ar[r]^{\eta_A}
    &
GA
}
\]
\end{definition}

\begin{definition}\label{def:dccat-trip}
We denote by $\cattrip$ the pre-equipment consisting of triposes, tripos
morphisms and tripos transformations, where the \regular\ 1-cells are the
regular tripos morphisms.
\end{definition}

It is straightforward to check that this data does indeed constitute a
pre-equipment. Furthermore, we have:
\begin{lemma}
The pre-equipment $\cattrip$ is geometric (in the sense of
Definition~\ref{def_preequipment}.\ref{def_preequipment_preequipment}).
\end{lemma}
\begin{proof}
Let $\effi\dashv\ggamma:\triq\to\trip$ be an adjunction between triposes
$\trip:\flx\to\flc$ and $\triq:\fly\to\fld$.
The fact that we have an adjunction means that we have the bidirectional rule
\[
\def\fCenter{\vdash}
\Axiom$\Phi\varphi\fCenter f^*\psi$
\doubleLine
\UnaryInf$\varphi\fCenter g^* (\Gamma\psi)$
\DisplayProof
\quad\qtext{with special case}\quad
\def\fCenter{\vdash}
\Axiom$\Phi\varphi\fCenter \psi$
\doubleLine
\UnaryInf$\varphi\fCenter \eta_C^* (\Gamma\psi)$
\DisplayProof
\]
for the entailment relations in the fibers of $\trip$ and $\triq$, where
$f:FC\to D$ and $g:C\to GD$ are conjugate to each other via $F\dashv G$,
$\varphi\in\trip_C$, $\psi\in\triq_D$, and in the special case we have
$f=\id_{FC}$.

Now the derivations
\[
\def\fCenter{\vdash}
\Axiom$\exists_f\varphi\fCenter \exists_f\varphi$
\UnaryInf$\varphi\fCenter f^*(\exists_f\varphi)$
\UnaryInf$\Phi\varphi\fCenter \Phi (f^*(\exists_f\varphi))\cong (Ff)^*(\Phi
(\exists_f\varphi))$
\UnaryInf$\exists_{Ff}(\Phi\varphi)\fCenter \Phi (\exists_f\varphi)$
\DisplayProof
\]
and
\[
\Axiom$\exists_{Ff}(\Phi\varphi)\fCenter\exists_{Ff}(\Phi\varphi)$
\UnaryInf$\Phi\varphi\fCenter (Ff)^*(\exists_{Ff}(\Phi\varphi))$
\UnaryInf$\varphi\fCenter
\eta_A^*(\Gamma((Ff)^*(\exists_{Ff}(\Phi\varphi))))\cong
f^*(\eta_B^*(\Gamma(\exists_{Ff}(\Phi\varphi))))$
\UnaryInf$\exists_f\varphi\fCenter \eta_B^*(\Gamma(\exists_{Ff}(\Phi\varphi)))$
\UnaryInf$\Phi(\exists_f\varphi)\fCenter \exists_{Ff}(\Phi\varphi)$
\DisplayProof
\]
where $f:A\to B$ in $\flc$ show that $\effi$ commutes with existential
quantification.
\end{proof}

\section{Q-toposes}\label{sec_q_toposes}

We introduce q-toposes as an intermediate step in our decomposition of the
\tttc.

Q-toposes are similar to quasitoposes%
\footnote{Quasitoposes are due to Penon~\cite{penon1973quasi}, my principal
reference is the
Elephant~\cite[A2.6]{elephant1}.
},
in particular they have a classifier for strong monomorphisms, but they have
less structure (that's why they have fewer letters in their name). Contrary to
quasi-toposes, q-toposes are not required to be locally cartesian closed (not
even cartesian closed), nor do they need to have all colimits. 

Fortunately it turns out that in order to get a working higher order logic,
neither of these features is needed, and it suffices to postulate the
quasitopos version of powersets. In the end, we can not entirely do without
colimits; we postulate pullback-stable effective quotients of \emph{strong
equivalence relations}, because we need them in a later proof.

The main difference between \emph{toposes} and \emph{quasitoposes} is that not
every monomorphism in a quasitopos corresponds to a predicate in the internal
logic. We rather restrict attention to a certain subclass of them, the
\emph{cocovers}. In the case of quasitoposes, there are two possible
definitions of cocovers, one by orthogonality, corresponding to the concept of
\emph{strong monomorphism}, and one by a factorization property, corresponding
to \emph{extremal monomorphisms}. These two definitions coincide in the
presence of pushouts, which we have in a quasitopos. For the weaker
notion of q-topos,
on the contrary, we do not require all pushouts, and we have to be careful
which definition to choose. It turns out that the extremality definition is too
weak, since the ensuing class of arrows might not even be stable under
pullbacks.
Therefore, we define cocovers (and at the same time \emph{covers}, which we
will need later) as follows.
\begin{definition}
 Let $\qcc$ be a category. 
\begin{itemize}
 \item 
Let $f:A\to B$, $g:X\to Y$ in $\qcc$.
We say that $f$ is \emph{left orthogonal} to $g$ (or alternatively that $g$ is
\emph{right orthogonal} to $f$), if for any commuting square
\[
\xymatrix{
A \ar[r]\ar[d]_f& X\ar[d]^g \\
B\dashed[ru]^h\ar[r] & Y
}\]
there exists a \emph{unique} $h:B\to X$ such that the two triangles commute.
\item
An epimorphism $e:A\to B$ is called a \emph{cover}, or a \emph{strong
epimorphism}, if it is left orthogonal to all monomorphisms.
\item
A monomorphism $m:X\to Y$ is called a \emph{cocover}, or a \emph{strong
monomorphism}, if it is right orthogonal to all epimorphisms.
\end{itemize}
\end{definition}
In the presence of equalizers, we can drop the assumption that $e$ is a
epimorphism from the definition of cover, since it follows from the
orthogonality condition. Similarly, in the presence of kernel pairs and
coequalizers of kernel pairs, we can drop the assumption that $m$ is a
monomorphism from the definition of cocover. We will denote covers by the
arrow `$\ecov$', and cocovers by the arrow `$\ecocov$'. 

It is easy to see that cocovers are stable under arbitrary pullbacks, and this
allows us to construct a presheaf%
\footnote{To avoid having to deal with size issues, we just assume that $\qcc$
is
small with respect to some universe. See also Section~\ref{suse_conv}.}
\[
 \subc : \qcc^\op\to\catset,
\]
which assigns to each object $A$ the set of isomorphism classes of cocovers
with common codomain $A$. 

In order to be able to phrase the definition of q-topos in a concise manner, we
send three more definitions ahead. For the first one, we assume that the reader
is acquainted  with the concept of an equivalence relation in a finite limit
category as a monomorphism $\rho: R\emono A\times A$ with certain properties.
The precise definition can be found e.g.\ in
\cite[Definition~1.3.6]{elephant1}. 
\begin{definition}
Let $\qcc$ be a category with finite limits.
\begin{itemize}
\item
A \emph{strong equivalence relation} is an equivalence relation $\rho: R\ecocov
A\times A$ which is a strong monomorphism.
\item
An \emph{effective quotient} of an equivalence relation
$\rho=\langle\rho_1,\rho_2\rangle: R\emono A\times A$ in $\qcc$ is a
coequalizer $e:A\eepi Q$ of $\rho_1,\rho_2:R\rightrightarrows A$ whose kernel
pair is $\rho_1,\rho_2$.

(Observe that equivalence relations which have effective quotients are
necessarily strong.)
\item
An object $A$ of $\qcc$ is called \emph{exponentiating} if the presheaf
\[ \qcc(-\times X,A) \]
is representable for all objects $X$.
\end{itemize}
\end{definition}

Now the definition of q-topos is the following.

\begin{definition}\label{def:q-topos}
A q-topos is a finite limit category $\qcc$ such that
\begin{itemize}
\item
all presheaves
$
\subc(-\times A) : \qcc^\op\to\catset
$
are representable,
\item
$\qcc$ has effective quotients of strong equivalence relations, and
\item
regular epimorphisms are stable under pullback.
\end{itemize}
\end{definition}
An immediate consequence of the definition is:
\begin{lemma}\label{lem_q-topos-regular}
Any q-topos is a \emph{regular category}. This implies that the class of
regular epimorphisms coincides with the class of covers.
\end{lemma}
\begin{proof}
It is possible to define a regular category as a a category with finite
limits and coequalizers of kernel
pairs, where regular epimorphisms are stable under pullbacks (See e.g.
\cite[Definition~2.1.1]{borceux2}, except that there not all finite limits are
assumed). Kernel pairs are always strong equivalence relations, 
described as the pullback of $\delta_B:B\to B\times B$ along $f\times f$. 
hence the fact that every q-topos is regular follows directly from the
definition. The fact that regular epimorphisms coincide with covers is proved
e.g.\ in \cite[Proposition~2.1.4]{borceux2}.
\end{proof}

The following lemma shows how to rephrase the first condition of
Definition~\ref{def:q-topos} in a way that is closer to the internal language
which will be
introduced in the next section.
\begin{lemma}
In a finite limit category $\qcc$, the
presheaves
\[
 \subc(-\times A) : \qcc^\op\to\catset
\]
are \emph{all} representable if and only if the presheaf
\[
 \subc(-) : \qcc^\op\to\catset
\]
can be represented by an exponentiating object.
\qed
\end{lemma}
We will always denote the object representing $\subc(-)$ by $\Omega$ and the
element of $\subc(\Omega)$ which induces the natural isomorphism (generally
known as `generic predicate') by $\toptr$. 
\begin{lemma}\label{lem_cocov_reg}
\begin{enumerate}
 \item 
The domain of $\toptr:U\ecocov \Omega$ is terminal.
\item
The class of cocovers coincides with the class of regular monomorphisms.
\end{enumerate}
\end{lemma}
\begin{proof}
\emph{Ad 1.}
The postcomposition map $\toptr\circ - : \qcc(A,U)\to\qcc(A,\Omega)$ induces a
bijection between $\qcc(A,U)$ and arrows $f:A\to \Omega$ such that $f^*\toptr$
is an isomorphism. For each $A$, there is exactly one such arrow.

\emph{Ad 2.}
It is well known that in every category, regular monomorphisms are strong.
Conversely, every cocover $m:U\ecocov A$ is the equalizer of its classifying
map $\chi_m:A\to\Omega$ and $\toptr\,\circ \,!_A$.
\end{proof}

\subsection{The logic of q-toposes}

The fibers of the presheaf $\subc$ have a natural ordering which is the
inclusion of subobjects, i.e., for $m: U\ecocov A,n:V\ecocov A$
\[ m\vdash_A n\qtext{iff} \exists h:U\to V\qdot n h = m. \]

This ordering allows us to view $\subc$ (quotiented out by mutual inclusion) as
a presheaf of posets, and the  posetal fibration that we obtain from this
presheaf via the Grothendieck construction is isomorphic to the \emph{fibration
of cocovers}
\begin{equation}
\funs\qcc=\partial_1:\mathrm{coc}(\qcc)\to\qcc\label{eq_fib_cc},
\end{equation}
i.e.\ the quotient by equivalence of the full subfibration on cocovers of the
fundamental fibration $\partial_1:\commacat{\qcc}{\qcc}\to\qcc$.

The fibration of cocovers on a q-topos $\qcc$ is the home of the internal logic
of $\qcc$. It is easy to see that it has fiberwise finite meets, which are just
given by pullbacks.
In this section we show that it is even a tripos, which will enable us later to
define the forgetful functor from the pre-equipment of q-toposes to the
pre-equipment of triposes.
To do this we proceed in essentially the same way as Lambek and Scott do in
\cite{lambekscott86} for the definition of the internal language of a topos%
\footnote{Apparently this way of bootstrapping the internal logic of a topos is
originally due to Boileau and Joyal~\cite{boileau1981logique}.}.

To begin, we outline of the general strategy. First, we define a kind of
minimal internal language of a q-topos $\qcc$, whose term constructors are
projections, subset comprehension, element relation and equality. This internal
language is called \emph{type theory based on equality} in
\cite{lambekscott86}, we will refer to it as the \emph{core calculus}. 
We view terms of type $\Omega$ of the core calculus as predicates, and give a
intuitionistic sequent calculus style system of inference rules for them. The
fact that predicates are only special terms reflects the higher order nature of
the language.

Following Lambek and Scott, we will then show how to encode all propositional
connectives and quantifiers in the core calculus in such a way that the usual
rules of intuitionistic logic can be derived from the rules of the core
calculus. 
The claim that the fibration of cocovers is a tripos then follows almost
directly, because the terms of type $\Omega$ of the language correspond to the
predicates in the fibration of cocovers. The only thing that is missing is full
quantification (the language only gives quantification along projections), but
using a standard encoding, quantification along arbitrary morphisms can also be
obtained.

\medskip

The core calculus of a q-topos $\qcc$ is given in Table~\ref{table:core
calculus}. 
\begin{table}
\noindent\framebox[\textwidth]{\parbox{.96\textwidth}{
\medskip
\textbf{Types:}
\[A::= X \vertor 1 \vertor \Omega \vertor PA \vertor A\times A\qquad
X\in\objects{\qcc}\]

\smallskip

\textbf{Terms:}

\noindent
We use $\Delta$ to denote a context $x_1\vtp A_1,\ldots ,x_n\vtp A_n$ of typed
variables.
\medskip

\noindent\hspace{.5cm}
\def\fCenter{\csep}
\noindent\begin{tabular*}{\textwidth}{@{}c@{$\quad\qquad$}c}
\def\fCenter{}
\hspace{-.34cm}
\AX$\fCenter$
\def\fCenter{\csep}
\RightLabel{$\scriptstyle{(i=1,\dots ,n)}$}
\UI$\Delta \fCenter x_i : A_i$
\DP &
\def\fCenter{}
\hspace{-.34cm}
\AX$\fCenter$
\def\fCenter{\csep}
\UI$\Delta\fCenter \ast:1$ 
\DP\\[\bigskipamount]
\AX$\Delta,x\vtp A\fCenter \varphi[x]:\Omega$
\UI$\Delta\fCenter\{x | \varphi[x]\}:PA$
\DP &
\AX$\Delta\fCenter a:A$
\AX$\Delta\fCenter b:B$
\BI$\Delta\fCenter (a,b):A\times B$
\DP \\[\bigskipamount]
\AX$\Delta\fCenter a:A$
\AX$\Delta\fCenter M:PA$
\BI$\Delta\fCenter a\in M:\Omega$
\DP &
\AX$\Delta\fCenter a:A$
\AX$\Delta\fCenter a':A$
\BI$\Delta\fCenter a=a':\Omega$
\DP \\[\bigskipamount]
\AX$\Delta\fCenter a:X$
\RightLabel{$\, f\in\qcc(X,Y)$}
\UI$\Delta\fCenter f(a):Y$
\DP &
\end{tabular*}

\medskip

\def\fCenter{\vdash}
\textbf{Deduction rules:}

\smallskip

Here, $p_1,\dots,p_n,p,q$ denote terms of type $\Omega$, and $\Gamma$ denotes a
sequence of such terms.

\noindent
\begin{tabular}{@{}cc@{}}
\hspace{-.34cm}
\def\fCenter{}
\Axiom$\fCenter\phantom{|}$ 
\def\fCenter{\vdash}
\RightLabel{\vbox{\hbox{Ax}\vspace{1mm}\hbox{$\scriptstyle{(i=1,\dots ,n)}$}}}
\UnaryInf$\Delta\csep p_1,\dots,p_n\fCenter  p_i$
\DisplayProof
&
\Axiom$\Delta\csep \Gamma\fCenter  p$ 
\Axiom$\Delta\csep \Gamma,p \fCenter  q$
\RightLabel{Cut}
\BinaryInf$\Delta\csep \Gamma\fCenter  q$
\DisplayProof
\\[\bigskipamount]
\hspace{-.34cm}
\def\fCenter{}
\Axiom$\fCenter\phantom{|}$ 
\def\fCenter{\vdash}
\RightLabel{$=$R}
\UnaryInf$\Delta\csep \Gamma\fCenter  t=t$
\DisplayProof 
&
\Axiom$\Delta,x\vtp A\csep \Gamma\hspace{.916cm}\fCenter \varphi[x,x]$ 
\RightLabel{$=$L  					}
\UnaryInf$\Delta\csep \Gamma,s=t\fCenter \varphi[s,t]$
\DisplayProof
\\[\bigskipamount]
\Axiom$\Delta,x\vtp A\csep\Gamma\fCenter  p[x]=(x\in M)$
\RightLabel{P-$\eta$}
\UnaryInf$\Delta \csep \Gamma \fCenter \{ x| p[x]\}= M$
\DisplayProof
&
\hspace{-.34cm}
\def\fCenter{}
\Axiom$\fCenter\phantom{|}$ 
\def\fCenter{\vdash}
\RightLabel{P-$\beta$}
\UnaryInf$\Delta\csep \Gamma\fCenter (a\in\{x|p[x]\})=p[a]$
\DisplayProof 
\\[\bigskipamount]
\hspace{-.34cm}
\def\fCenter{}
\Axiom$\fCenter\phantom{|}$ 
\def\fCenter{\vdash}
\RightLabel{$1$-$\eta$}
\UnaryInf$\Delta\csep \Gamma\fCenter t=\ast$
\DisplayProof 
&
\Axiom$\Delta\csep\Gamma,p \fCenter q$ 
\Axiom$\Delta\csep\Gamma,q \fCenter p$
\RightLabel{Ext}
\BinaryInf$\Delta\csep\Gamma\fCenter p=q$
\DisplayProof 
\\[\bigskipamount]
\Axiom$\Delta\csep\Gamma\fCenter (a_1,a_2)=(a'_1,a'_2)$ 
\RightLabel{\vbox{\hbox{$\times$-$\beta$}\vspace{1mm}\hbox{$\scriptstyle(i=1,
2)$}}}
\UnaryInf$\Delta\csep\Gamma\fCenter a_i=a_i'$
\DisplayProof 
&
\Axiom$\Delta, x\vtp A, y\vtp B\csep \Gamma,t=(x,y)\fCenter p[t]$
\RightLabel{$\times$-$\eta$}
\UnaryInf$\Delta\csep \Gamma\hspace{1.59cm}\fCenter p[t]$
\DisplayProof 
\end{tabular}
}}
\caption{The core calculus}\label{table:core calculus}
\end{table}

The interpretation of the language in $\qcc$ is given as follows.
Types are inductively interpreted by objects of $\qcc$, where $\sem{X} = X$,
$\sem{1} = 1$, $\sem{\Omega}=\Omega$, $\sem{PA} = \Omega^{\sem{A}}$, and
$\sem{A\times B} = \sem{A}\times \sem{B}$. 
Contexts are interpreted as cartesian products of their constituents, and terms
are interpreted by suitably typed morphisms as follows:
\begin{align*}
\sem{\Delta\csep x_i} & \;=\; \pi_i\\
 \sem{\Delta\csep\ast} &\;=\;\; !_{\sem{\Delta}}\\
\sem{\Delta\csep\{x | \varphi[x]\}} &\;=\; \Lambda(\sem{\Delta,x\csep
\varphi[x]})\\
\sem{\Delta\csep a\in M}&\;=\;\ve_\sem{A}\circ\langle\sem{\Delta\csep
M},\sem{\Delta\csep a}\rangle\\
\sem{\Delta\csep (a,b)} & \;=\; \langle\sem{\Delta\csep a},\sem{\Delta\csep
b}\rangle\\
\sem{\Delta\csep s=t} & \;=\; 
\chi_=\circ\langle\sem{\Delta\csep s},\sem{\Delta\csep t}\rangle\\
\sem{\Delta\csep f(a)} &\;=\; f\circ \sem{\Delta\csep a}
\end{align*}
Here $\pi_i$ denotes the suitable projection, $!_{\sem{\Delta}}$ denotes the
terminal projection, $\Lambda:\qcc(\sem{\Delta,x\vtp
A},\Omega)\to\qcc(\sem{\Delta},\Omega^{\sem{A}})$ is the exponential transpose
operation, $\ve_\sem{A}:\Omega^\sem{A}\times\sem{A}\to\Omega$ is evaluation,
and $\chi_=:A\times A\to\Omega$ is the classifying map of the diagonal
$\delta_A:A\to A\times A$, i.e., the map that corresponds to
$\delta_A\in\subc(A\times A)$ via the bijection $\subc(A\times
A)\cong\qcc(A\times A,\Omega)$ ($\delta_A$ is indeed in $\subc(A\times A)$,
since it is a split mono and therefore necessarily a cocover).

We make some comments on the deduction rules.
The first two rules, (Ax) and (Cut), are just the usual structural rules known
from sequent calculus style systems. ($=$L) and ($=$R) are equivalent to the
usual rules for equality in first order logic. The rules ($P$-$\beta$) and
($P$-$\eta$) correspond to the $\beta$- and $\eta$-rules of lambda calculus
if we identify subset comprehension and $\in$ with abstraction and application.
Together with the rule (Ext), which says that two predicates which are
equivalent are already equal, $P$-$\eta$ can be read as saying that \emph{two
sets are already equal if they have the same elements}, which is a more
familiar phrasing of set-theoretic extensionality. Finally, there are a
$\beta$- and an $\eta$ rule for pairing, and an $\eta$-rule for the unit type.

For the core calculus, we have a substitution lemma and a soundness theorem,
which we state without proof because the proofs are standard and consist mainly
of unfolding of definitions.
\begin{lemma}[Substitution lemma]\label{lem_substilem}
For well formed terms ${\Delta\csep s_i\ttp B_i}, 1\leq i\leq n$, and
${x_1\vtp
B_1,\dots,x_n\vtp B_n\csep t : C}$, we have
\[
 \sem{\Delta\csep t[s_1/x_1,\dots,s_n/x_n]}=\sem{x_1,\dots,x_n\csep
t}\circ\langle \sem{\Delta\csep s_1},\dots,\sem{\Delta\csep s_n}\rangle.
\]
\qed
\end{lemma}

\begin{theorem}[Soundness theorem]
If a sequent $\Delta\csep p_1,\dots,p_n\vdash q$ is derivable in the
core
calculus, then we have
\begin{equation}
\sem{\Delta\csep p_1}^*\toptr\wedge\dots\wedge\sem{\Delta\csep
p_n}^*\toptr\vdash_{\sem{\Delta}}\sem{\Delta\csep q}^*\toptr\tag{$\ast$}
\end{equation}
in $\subc(\sem{\Delta})$. \qed
\end{theorem}
If the relation $(\ast)$ holds, we also say that \emph{the judgment
$\Delta\csep p_1,\dots,p_n\vdash q$ holds in $\qcc$}. 

The connectives of predicate logic can be encoded in the core calculus as
follows:
\begin{align*}
\top & \defcon \ast=\ast\\
p\wedge q &\defcon (p,q)=(\top,\top)\\
p\Rightarrow q &\defcon p\wedge q = p\\
\forall x\vtp A\qdot p[x] &\defcon \{x|p[x]\}=\{x|\top\}\\
\bot &\defcon \forall z\vtp\Omega\qdot z\\
p\vee q &\defcon \forall z\vtp\Omega\qdot (p\Rightarrow z)\wedge (q\Rightarrow
z)\Rightarrow z\\
\exists x\vtp A\qdot p[x] &\defcon \forall z\vtp\Omega\qdot(\forall x\vtp
A\qdot p[x]\Rightarrow z)\Rightarrow z
\end{align*}
It can now be derived from the rules of the core calculus that the such defined
logical connectives validate the rules of intuitionistic predicate logic. For a
proof of this, we refer to Lambek and Scott's book (alternatively, the reader
may prove this herself as an instructive exercise).

Here are some basic principles that can be transferred directly from toposes to
q-toposes:

\begin{lemma}\label{lem:int-lang}
\begin{enumerate}
\item\label{lem:int-lang:equal}
 Two parallel arrows $f,g:A\to B$ in $\qcc$ are equal if and only if
$\ilbracks{x\csep
\vdash f(x)=g(x)}$ holds in the internal logic.
\item\label{lem:int-lang:mono}
$f:A\to B$ is a monomorphism iff $f(x)=f(y)\vdash x=y$
holds in
$\qcc$.
\item\label{lem:int-lang:epi}
 $f:A\to B$ is an epimorphism iff $\;y\vtp B\!\csep\vdash\exists x\qdot f(x)=y$
holds
in $\qcc$.
\end{enumerate}
\end{lemma}
\begin{proof}
\emph{Ad \ref{lem:int-lang:equal}.} Straightforward, since the interpretation
of $\ilbracks{x\csep f(x)=g(x)}$ is the equalizer of $f$ and $g$.

\emph{Ad \ref{lem:int-lang:mono}.} The interpretation of $\ilbracks{x,y\csep
f(x)=f(y)}$
is the kernel pair of $f$, and a morphism is monic iff its kernel pair is the
diagonal.

\emph{Ad \ref{lem:int-lang:epi}.} Assume that $\;y\vtp B\!\csep\vdash\exists
x\qdot f(x)=y$ holds. Let $h,k:B\to C$ be arbitrary and assume that $hf = kf$.
Then
the deduction
\def\fCenter{\vdash}
\begin{prooftree}
\AX$y\csep\fCenter\exists x\qdot fx =y$
\AX$x\csep\fCenter h(f(x))=k(f(x))$
\UI$x,y\csep fx = y \fCenter h(y)=k(y)$
\BI$y\csep\fCenter h(y)=k(y)$
\end{prooftree}
establishes the claim.

Conversely, assume that $f$ is an epimorphism. $f$ obviously equalizes the
classifying maps of the predicates $\ilbracks{y\vtp B\csep \top}$ and
$\ilbracks{y\vtp B\csep \exists x\qdot fx = y}$,
and since it is epic they are both equal, whence the second is valid in $\qcc$.
\end{proof}

Note that statement \ref{lem:int-lang:epi} stands in contrast to regular
categories, where the judgment $y\csep\vdash \exists x\qdot fx=y$
characterizes the \emph{regular} epimorphisms. This discrepancy is due to the
fact that the internal logic of a regular category has \emph{all} monomorphisms
as predicates, whereas we only consider the strong ones in q-toposes.

\medskip

In \cite{elephant1}, Johnstone shows that every quasitopos is coregular, i.e.,
its opposite category is regular. For q-toposes, this is too much to hope for,
since we do not even require all colimits. However, we can still prove one of
the most important consequences of coregularity --- the existence of a
epi/cocover factorization system.

\begin{lemma}
Let $f:A\to B$ in $\qcc$. The cocover $m:U\ecocov A$ that is classified by the
predicate $\ilbracks{y\csep \exists x\qdot f(x)=y}$ is the minimal cocover
through which
$f$ factors, and in the corresponding factorization $f=me$, $e$ is an
epimorphism.
In particular, the class 
of epimorphisms and the class 
of cocovers together form a factorization
system (see~\cite{freyd-kelly-cats_of_cont_functors,wiki:facsys}) on $\qcc$.
\end{lemma}
\begin{proof}
The minimality condition is just a rephrasing of the universal property of
existential quantification.  If $e:A\to U$ were not an epimorphism, it would
factor through the monomorphism $m':U'\to U$ classified by $\ilbracks{y\csep
\exists
x\qdot e(x)=y}$, which would be non-maximal by the previous lemma. Then $f$
would factor through $mm'$, contradicting the minimality of $m$.

To establish that the epimorphisms together with the cocovers constitute a
factorization system, it now remains to show that the maps that are left
orthogonal to all cocovers are precisely the epimorphisms. This follows from
the \emph{retract argument}: Let $g: C\to D$ be a morphism that is left
orthogonal to all cocovers, and let $f=kh$ be its factorization into an
epimorphism followed by a cocover. We obtain a square
\[
 \xymatrix{
C\ar[d]_g\depi[r]^h & X\cocov[d]^k \\
D\ar[r]_\id\dashed[ru]_l & D
},
\]
which has a lifting $l$ because $g$ is left orthogonal to all cocovers. It is
now easy to see that $g$ is a retract of $h$ and this implies that $g$ is an
epimorphism, because epimorphisms are closed under retracts.
\end{proof}

It is remarkable that to prove the existence of the factorization system, we
need quite heavy machinery, in particular the higher order internal logic with
its polymorphically defined existential quantification. 

\begin{lemma}\label{lem:pullback}
\begin{enumerate}
 \item 
If the square $(\dagger)$ is a pullback in $\qcc$, then the judgments
$(i)$--$(iii)$ hold in the internal logic. (The converse is not true in
general.)
\[(\dagger)\vcenter{\xymatrix{
P \ar[r]^q\ar[d]_p & B\ar[d]^g \\
A\ar[r]_f & C
}}\quad
\begin{array}{ll@{}l@{}l}
(i)&\!z&\csep &\vdash f(pz)=g(qz)\\
(ii)&\!x,y &\csep fx=gy&\vdash\exists z\qdot (pz,qz)=(x,y)\\
(iii)&\!z,z'&\csep (pz,qz)=(pz',qz')&\vdash z=z'
\end{array}
\]
\item\label{lem:pullback:stable}
Epimorphisms are stable under pullback in a q-topos.
\end{enumerate}
\end{lemma}
\begin{proof}
\emph{Ad 1.}
Assume that the square is a pullback. $(i)$ just states that the square
commutes. $(ii)$ holds, because $\ilbracks{x,y\csep fx=gy}$ as well as
$\ilbracks{x,y\csep\exists
z\qdot (pz,qz)=(x,y)}$ classify the cocover $P\ecocov A\times B$. $(iii)$
expresses the fact that $p$ and $q$ are jointly monic.

\emph{Ad 2.}
This follows from 1 and the characterization of epimorphisms in
Lemma~\ref{lem:int-lang}.\ref{lem:int-lang:epi}.
\end{proof}

\begin{lemma}\label{lem:fib-cocov-trip}
The fibration $\funs\qcc$ of cocovers on a q-topos $\qcc$ is a tripos.
\end{lemma}
\begin{proof}
The internal language gives us the propositional structure and quantification
along projections. Quantification along arbitrary morphisms can be encoded as
follows. For $f:A\to B$ and $\varphi: A\to \Omega$, we set
\begin{align*}
\forall_f(\varphi)(y)&\defcon\forall x\vtp A\qdot f(x)=y\Rightarrow
\varphi(x)\qquad\text{and}\\
\exists_f(\varphi)(y)&\defcon\exists x\vtp A\qdot f(x)=y\wedge \varphi(x).
\end{align*}
The Beck-Chevalley condition for pullback squares of the form
\[\xymatrix{
A\times C\ar[r] \ar[d]_{f\times C} & A\ar[d]^f\\
B\times C\ar[r] & B
}\]
(where we quantify along the projections) follows directly from the
Substitution Lemma~\ref{lem_substilem}. For general pullback squares, it
follows from Lemma~\ref{lem:pullback} and some calculations in the internal
logic. Finally, it is easy to see that the power objects of the q-topos give us
power objects in the fibration in the tripos sense.
\end{proof}

\subsection{The subtopos of coarse objects}

In Lemma~\ref{lem:int-lang}, we saw that the internal logic of a q-topos is
powerful enough to detect equality of arrows. However, the internal logic lacks
another important feature: it is not capable to distinguish isomorphisms from
maps which are monomorphisms and epimorphisms at the same time. This follows
from the following lemma.
\begin{lemma}
 Let $f:A\emonepi B$ be monic as well as epic. Then the induced map
$f^*:(\funs\qcc)_B\to(\funs\qcc)_A$ 
we do
is an isomorphism of posets.
\end{lemma}
\begin{proof}
Using the internal language, the map $\subc(f)$ can be written as
\[
 (\funs\qcc)_B\ni\varphi\quad\mapsto\quad \ilbracks{a\csep \varphi(fa)}.
\]
A map in the converse direction is given by existential quantification:
\[
(\funs\qcc)_B\ni\psi\quad\mapsto\quad \ilbracks{b\csep\exists a\qdot fa=b\wedge
\psi(a)}
\]
Using the characterizations of monomorphisms and epimorphisms of
Lemma~\ref{lem:int-lang}, it is easy to verify that these two maps are inverse
to each other.
\end{proof}

So in a sense the arrows which are monic and epic at the same time disclose a
mismatch between the category and the internal logic. This can be seen as a
motivation for the following definition of \emph{coarse objects}, which are
just as blind as the internal logic, so that the correspondence between
category and internal logic is restored if we restrict to the full subcategory
on the coarse objects.

Coarse objects are also considered for quasitoposes, and the treatment here is
a variation of the presentation in \cite{elephant1} for quasitoposes.

\begin{definition}
An object $C$ of a q-topos $\qcc$ is called coarse, if  for each morphism
$f:A\emonepi B$ which is monic and epic at once and all morphisms $g:A\to C$,
there exists a morphism $h:B\to C$ such that $hf=g$.
\end{definition}

Because the arrow $f$ in the previous definition is an epimorphism, the
mediating arrow $h$ is automatically unique.

\begin{lemma}[Properties of coarse objects]\label{lem:coarse-objects}
\begin{enumerate}
\item\label{lem_coarse-objects_mono-epi-out}
If $C$ is coarse and $f:C\emonepi A$ is monic as well as epic, then it is
already an isomorphism.
\item \label{lem_coarse-objects_mono-out}
If $C$ is coarse and $m:C\emono A$ is a monomorphism, then $m$ is already a
cocover.
\item \label{lem_coarse-objects_subobj}
If $C$ is coarse and $m:U\ecocov C$ is a cocover, then $U$ is coarse.
\item \label{lem_coarse-objects_prods}
Finite products of coarse objects are coarse.
\item \label{lem_coarse-objects_pow-objs}
 For every object $A$ of $\qcc$, its power object $PA$ is coarse.
\item \label{lem_coarse-objects_reflective}
The full subcategory $\funt\qcc$ of  $\qcc$ on the coarse objects is
reflective.

(In the following we denote this reflection by $J_\qcc\dashv
I_\qcc:\funt\qcc\to\qcc$.)
\item \label{lem_coarse-objects_j-regular}
The reflection functor $J_\qcc:\qcc\to\funt\qcc$ preserves finite limits and
epimorphisms.
\item \label{lem_coarse-objects_topos}
$\funt\qcc$ is a topos.
\item \label{lem_coarse-objects_i-preserves-epis}
The embedding functor $I_\qcc:\funt\qcc\to \qcc$ maps epimorphisms to
epimorphisms.
\end{enumerate}
\end{lemma}
\begin{proof}
\emph{Ad \ref{lem_coarse-objects_mono-epi-out}.} By coarseness of $C$, there
exists an arrow $g:A\to C$ such that $gf=\id_C$. Because $f$ is an epimorphism,
it follows that $fg=\id_{A}$.

\smallskip
\emph{Ad \ref{lem_coarse-objects_mono-out}.} Let $\tilde{m}e=m$ be the
factorization of $e$ into an epimorphism and a cocover. Then $e$ is a
monomorphism as well as an epimorphism, and by \emph{1.}\ it is already an
isomorphism.

\smallskip
\emph{Ad \ref{lem_coarse-objects_subobj}.} Let $f:A\to B$ be monic and epic,
and $g:A\to U$. Then by coarseness of $C$, there exists a  map $h:B\to C$ with
$hf=mg$, and by orthogonality, there exists a map $k:B\to U$ such that $kf=g$
and $mk=h$.
\[\xymatrix{
A\monepi[d]_f\ar[r]^g & U\cocov[d]^m \\
B\dashed[r]_h\dashed[ru]_k & C 
}\]

\smallskip
\emph{Ad \ref{lem_coarse-objects_prods}.} To extend an arrow $\langle
g_1,\dots,g_n\rangle:A\to C_1\times\dots\times C_n$ along a monic-epic arrow
$f:A\emonepi B$, simply extend the components $g_i$ individually.

\smallskip
\emph{Ad \ref{lem_coarse-objects_pow-objs}.} Let $f:A\emonepi B$ be monic-epic
and let $g:A\to PD$. The lifting of $g$ along $f$ can be elegantly expressed as
$\ilbracks{b\csep\{d|\exists a\qdot f(a)=b\wedge d\in g(a)\}}$. To see that
this
validates the required equation, simply substitute and simplify.

\smallskip
\emph{Ad \ref{lem_coarse-objects_reflective}.} We have to give for each $A$ a
coarse object $\crefl{A}=J_\qcc(A)$ (we use the notations $\crefl{A}$ and
$J_\qcc A$ interchangeably) and a morphism $\eta_A:A\to \crefl{A}$ such that
for all coarse $C$ and all $f:A\to C$ there exists a unique $g:\crefl{A}\to C$
with $g\eta_A = f$.

Consider the morphism $\ilbracks{x\csep \{y|y=x\}}:A\to PA$. Using
Lemma~\ref{lem:int-lang}.\ref{lem:int-lang:mono}, we deduce that it is a
monomorphism, and construct its epi/cocover factorization
$A\stackrel{\eta_A}{\emonepi} \crefl{A}\ecocov PA$. It follows
from~\ref{lem_coarse-objects_subobj} and~\ref{lem_coarse-objects_pow-objs} 
that $\crefl{A}$ is coarse, and the required
universal property of $\eta_A$ follows directly from the fact that it is a
mono-epi.

\smallskip
\emph{Ad \ref{lem_coarse-objects_j-regular}.}
Epimorphisms are preserved because they can be characterized in terms of
colimits, and $J$ is a left adjoint. 

For the finite limits, we show that $J$ preserves the terminal object, binary
products and equalizers.

We have already seen that the terminal object is coarse, hence $\crefl{1}=1$,
i.e., the terminal object is preserved.

For products, we know from \ref{lem_coarse-objects_prods} that
$\crefl{A}\times\crefl{B}$ is coarse  and thus a product of $\crefl{A}$ and
$\crefl{B}$ in $\funt\qcc$. Hence we have to show that $\crefl{A\times
B}\cong\crefl{A}\times\crefl{B}$. To show this, it is sufficient to find a
monic-epic arrow $\eta'_{A\times B}:A\times B\emonepi
\crefl{A}\times\crefl{B}$, since  this makes $\crefl{A}\times\crefl{B}$ into a
coarse hull of $A\times B$ having the universal property that was described in
\ref{lem_coarse-objects_reflective}.
The obvious candidate arrow is the product $\eta_A\times\eta_B$ of the
universal arrows for $A$ and $B$. It can be decomposed as $\eta_A\times \eta_B
= \eta_A\times \id_B\circ \id_A\times \eta_B$, and $\eta_A\times\id_B$ and
$\id_A\times \eta_B$ are monic and epic since they are pullbacks of $\eta_A$
and $\eta_B$, respectively, and epimorphisms are pullback stable by
Lemma~\ref{lem:pullback}.\ref{lem:pullback:stable}.

Finally, consider a pair $f,g:A\to B$ of parallel arrows, with equalizer
$m:U\ecocov A$. We want to show that $\crefl{m}$ equalizes $\crefl{f}$ and
$\crefl{g}$. Let $h:X\to A$ such that $\crefl{g}h=\crefl{f}h$, form the
pullback $k$ of $h$ along $\eta_A$, and consider the diagram
\[
 \vcenter{\xymatrix@R-1mm{
X\ar[rrd]^h\dashed[rd]_w \\
& \crefl{U}\mono[r]_{\crefl{m}} &\crefl{A}\ppair{r}{\crefl{f}}{\crefl{g}} &
\crefl{B}\\
& U\mono[r]^m\monepi[u]^{\eta_U} & A\ppair{r}{f}{g}\monepi[u]^{\eta_A} &
B\monepi[u]_{\eta_B}\\
Y\ar[rru]_k\monepi[uuu]^l\dashed[ru]^v
}}
\]
where $l$ is the pullback of $\eta_A$ along $h$ (and thus monic and epic).
We can derive in the internal logic that $\crefl{m}$ is a monomorphism from the
facts that $\crefl{m}\eta_U$ is a monomorphism and $\eta_U$ is an epimorphism.
Furthermore we have $\eta_B fk = \crefl{f} h l = \crefl{g} h l = \eta_B g k$,
and because $\eta_B$ is a monomorphism, this implies $gk = fk$. This induces a
mediating arrow $v:Y\to U$, and by coarseness we can lift $\eta_U v$ along $l$
to obtain $w$. Doing some arrow chasing we conclude $\crefl{m}w=h$, and $w$ is
the unique such because $\crefl{m}$ is a monomorphism. Hence, $\crefl{m}$ is
indeed an equalizer of $\crefl{f}$ and $\crefl{g}$. 

\smallskip

\emph{Ad \ref{lem_coarse-objects_topos}.} It follows from
\ref{lem_coarse-objects_subobj} and \ref{lem_coarse-objects_prods} that
$\funt\qcc$ has finite limits, because equalizers are cocovers. The powersets
are coarse by \ref{lem_coarse-objects_pow-objs}, and from
\ref{lem_coarse-objects_mono-out} and \ref{lem_coarse-objects_subobj} it
follows that the maps that are classified by arrows of type $A\to PB$ are
precisely the monomorphisms $m:U\emono A\times B$ with coarse $U$.

\smallskip
\emph{Ad \ref{lem_coarse-objects_i-preserves-epis}.}
Let $e:A\to B$ be an epimorphism in $\funt\qcc$. We take its epi-cocover
factorization $A\eepi D\ecocov C$ in $\qcc$. Being a regular subobject of a
coarse object, $D$ is also coarse, and thus the factorization is also an
epi-mono factorization in $\funt\qcc$. Since $e$ is an epimorphism and
$\funt\qcc$ is balanced (as a topos), the map $D\ecocov C$ is an isomorphism.
Thus, $e$ is an epimorphism in $\qcc$.
\end{proof}

To conclude the section, we explain in which way we want to view q-toposes as a
pre-equipment.
\begin{definition}[The pre-equipment $\catqtop$]\label{def:dccat-qtop}
The underlying 2-category of $\catqtop$ consists of q-toposes, finite limit
preserving functors and arbitrary transformations. 

The \regular\ 1-cells are the functors that preserve epimorphisms and regular
epimorphisms.
\end{definition}
The pre-equipment $\catqtop$ is clearly \tworeg, since epimorphisms as well as
regular epimorphisms may be characterized in terms of colimits, and these are
preserved by left adjoints.

The attentive reader may have noticed that the definition of regular
1-cells 	does not mention
cocovers. Should we not demand that they are also preserved? --- It turns out
that this comes for free, because we proved in~\ref{lem_cocov_reg} that the
cocovers coincide with the regular monomorphisms, and those are preserved by
any finite limit preserving functor.

\section{The \tttc}\label{sec_tttc}

To each topos, we can associate a tripos --- its \emph{subobject fibration} ---
in a 2-functorial manner. The present section is dedicated to proving that the
such defined 2-functor has a \emph{special left biadjoint}. In
\cite{master-thesis}, I gave a very technical direct proof
of this statement. Later I found a substantial simplification, which consists
in decomposing the forgetful functor into two steps, the intermediate stage
being q-toposes. We will now see how this allows a simple description of
the special left biadjoints.

Sections \ref{suse:adj-f-s} and \ref{suse:adj-t-u} are devoted to the
description of the special biadjunctions between triposes and q-toposes, and
between q-toposes and toposes, respectively.

To simplify notation, we will in the  following use the variables $\eta$ and
$\ve$ in an ambiguous sense, denoting unit and counit of whatever adjunction is
at hand. We will use the convention that boldface $\pseta,\psve$ denote unit
and counit of 2-dimensional adjunctions, i.e., the special biadjunctions in our
case, whereas we use normal $\eta,\ve$ for 1-dimensional adjunctions.

\subsection{The biadjunction \texorpdfstring{$\funf\dashv\funs$}{F,S}
between\\ triposes and q-toposes}
\label{suse:adj-f-s}

In this section, we define a special biadjunction
\[
 (\funf\dashv\funs:\catqtop\to\cattrip,\pseta,\psve,\nu,\mu)
\]
between the pre-equipments of triposes and q-toposes.

The definitions and verifications of well-definedness of the constituents are a
bit long-winded, therefore we devote a individual subsection to each of them.

\subsubsection{The special right biadjoint functor
\texorpdfstring{$\funs$}{S}}

The forgetful functor
\[
\funs:\catqtop\to\cattrip
\]
assigns to each q-topos $\qcc$ its fibration
$\funs\qcc=\partial_1:\funcoc(\qcc)\to\qcc$ of cocovers. We defined this
fibration in \eqref{eq_fib_cc}  and proved in Lemma~\ref{lem:fib-cocov-trip}
that it is indeed a tripos. 

To a finite limit preserving functor $F:\qcc\to\qcd$, $\funs$ assigns the
tripos transformation
\[
\vcenter{\xymatrix@C+5mm{
\funcoc(\qcc) \ar[r]^{\funcoc(F)} \ar[d]_{\partial_1} &
\funcoc(\qcd) \ar[d]^{\partial_1} \\
\qcc\ar[r]_F & 
\qcd
}}.
\]
Here, $\funcoc(F)$ denotes the functor that maps the cocover $m:U\ecocov A$ to
$Fm:FU\ecocov FA$, which is again a cocover by the remark after
Definition~\ref{def:dccat-qtop}. 

To verify that this indeed defines a tripos morphism, we have to check the four
conditions in Definition~\ref{def:tripos-morphism}. Clearly the square
commutes, and cartesian arrows are preserved by $\funcoc(F)$, because they are
just pullback squares. Finite limits are preserved by $F$ by assumption, and
fiberwise finite meets are preserved again because they are given by pullbacks.

The action of $\funs$ on 2-cells is easy again, a natural transformation
$\eta:F\to G:\qcc\to \qcd$ is just mapped to itself and we leave it to the
reader to verify that $\eta$ constitutes a tripos transformation from
$(F,\funcoc(F))$ to $(G,\funcoc(G))$.

It is straightforward to verify that the previously described constructions all
commute with compositions and identities on the nose and thus $F$ is a
strict functor. We want it to be a special functor, and thus we have to check
that it maps \regular\ 1-cells in $\catqtop$ to \regular\ 1-cells in
$\cattrip$,
i.e., that $\funcoc(F)$ preserves cocartesian arrows whenever $F$ preserves
epimorphisms and regular epimorphisms. But cocartesian arrows in
$\funcoc(C)$ and $\funcoc(D)$ are just
squares
\[
\cxymatrix{
U\depi[r]^h\cocov[d]_m & V\cocov[d]^n \\
A\ar[r]^f & B
}
\]
where $m,n$ are cocovers and $h$ is an epimorphism, which are all preserved by
$F$.

This completes the description of $\funs$.

\subsubsection{The special left biadjoint functor \texorpdfstring{$\funf$}{F}}

The construction of the left biadjoint functor
\[
\funf:\cattrip\to\catqtop
\]
from triposes to q-toposes is more involved than that of $\funs$, which is to
be expected because it is a kind of `free' construction.

The object part of $\funf$ is close to --- but not quite\footnote{The
construction here has a more restrictive notion of morphism than the classical
\tttc, and in particular does not produce toposes but only q-toposes.}
 --- what is traditionally known as the `\tttc' (i.e. the construction of the
category $\flc[\trip]$ from a tripos $\trip$, as described in
\cite{hjp80,pitts81}). Concretely, the category $\funf\trip$ for a tripos
$\trip:\flx\to\flc$ is given as follows.
\begin{definition}[The category $\funf\trip$]\label{def:funq-trip} 
Let $\trip:\flx\to\flc$ be a tripos.
\begin{itemize}
 \item The objects of $\funf\trip$ are pairs $(C,\rho)$, where
$C\in\objects{\flc}$ and $\rho\in\trip_{C\times C}$ is a \emph{partial
equivalence relation} (i.e., symmetric and transitive) in the logic of $\trip$.
\item Morphisms $f:(C,\rho)\to (D,\sigma)$ in $\funf\trip$ are given by 
morphisms $f:C\to D$ in $\flc$ which satisfy
\[
 x,y\csep \rho(x,y)\vdash\sigma(fx,fy)
\]
in the logic of $\trip$.
Two such morphisms $f,g:C\to D$ are identified as morphisms from $(C,\rho)$ to
$(D,\sigma)$ if
\[
x\csep \rho(x,x)\vdash \sigma(fx,gx)
\]
holds. (More concisely, we can define $\funf\trip(\cro,\dsi)=\flc(C,D)/\sim$,
where $\sim$ is an (external) partial equivalence relation defined by $f\sim g$
iff $\rho(x,y)\vdash\sigma(fx,gy)$.)
\item
Composition and identities are inherited from $\flc$.
\end{itemize}
\end{definition}

It is straightforward to see that $\funf\trip$ is a well defined category. 
In Lemma~\ref{lem:fp-qtop}, we will prove that it is even always a q-topos, but before we introduce some terminology.

\begin{definition}
Let $\trip:\flx\to\flc$ be a tripos.
\begin{enumerate}
 \item Let $f:C\to D$ and let $\rho\in\trip_{C\times C},\sigma\in\trip_{D\times
D}$ be partial equivalence relations.
If the judgment 
\[
 x,y\csep \rho(x,y)\vdash\sigma(fx,fy)
\]
holds (i.e., if $f$ represents a morphism from $(C,\rho)$ to $(D,\sigma)$ in
$\funt\trip$), we say that $f$ is \emph{well defined with respect to $\rho$ and
$\sigma$}.

\item
Let $(C,\rho)$ be an object in $\bftr$. We call a predicate $\varphi\in\trip_C$
\emph{compatible with $\rho$}, if the judgments
\[x\csep \varphi(x)\vdash\rho(x,x)\qtext{and}x,y\csep
\varphi(x),\rho(x,y)\vdash\varphi(y)\]
hold in $\trip$ (intuitively, this means that $\varphi$ is a union of
equivalence classes of $\rho$). In Pitts'~\cite{pitts81} terminology, a
predicate on $C$ compatible with $\rho$ is precisely a strict unary relation on
$(C,\rho)$.
\end{enumerate}
\end{definition}

\begin{lemma}\label{lem:fp-qtop}
Let $\trip:\flx\to\flc$ be a tripos.
\begin{enumerate}
\item\label{lem:fp-qtop:flims}
The category $\bftr$ has finite limits.
\item  \label{lem:fp-qtop:omega}
The formula
\[
 \pereqv(p,q)\defpred \triptr(p)\leftrightarrow\triptr(q)
\]
defines a partial equivalence relation on $\prop\in\objects{\flc}$.
\item \label{lem:fp-qtop:epis}
$e:(C,\rho)\to (D,\sigma)$ is an epimorphism in $\funf\trip$ iff 
\[
y\csep \sigma(y,y)\vdash\exists x\qdot \rho(x,x)\wedge\sigma(ex,y)
\]
holds in $\trip$.
 \item \label{lem:fp-qtop:compat-to-sub}
If $\varphi\in\trip_D$ is a predicate which is compatible with $\sigma$, then
\[
 \sigma|_\varphi(x,y)\defpred\sigma(x,y)\wedge\varphi(x)
\]
is a partial equivalence relation on $D$, and
\[
 \id:(D,\sigma|_\varphi)\ecocov(D,\sigma)
\]
is a cocover in $\funf\trip$.
\item \label{lem:fp-qtop:mor-to-compat}
For any morphism $f:(C,\rho)\to (D,\sigma)$ in $\funf\trip$, the predicate
\[
 \psi(y)\defpred \exists x\qdot \rho(x,x)\wedge\sigma(fx,y)
\]
on $D$ is compatible with $\sigma$. Furthermore, $f:(C,\rho)\to (D,\sigma)$
factors through $(D,\sigma|_\psi)\ecocov(D,\sigma)$, giving rise to a
epi-cocover factorization.
\item\label{lem:fp-qtop:compat-sub-bij}
The assignments
\begin{align*}
 \varphi &\quad\mapsto\quad\bigl((C,\rho|_\varphi)\ecocov (C,\rho)\bigr)\\
 \bigl((U,\nu)\ecocov(C,\rho)\bigr)&\quad\mapsto\quad \ilbracks{y\csep \exists
x\qdot
\nu(x,x)\wedge\rho(mx,y)}
\end{align*}
constitute a bijection between the isomorphism classes of cocovers with
codomain $(C,\rho)$ and predicates on $C$ compatible
with $\rho$.
\item\label{lem:fp-qtop:reindexing}
The bijection established in the previous item provides a more convenient
characterization of the fibration of cocovers on $\bftr$. In this
representation, the reindexing of a predicate $\varphi\in\trip_D$ that is
compatible with $\sigma$ along $f:(C,\rho)\to(D,\sigma)$ is given by

\[
x\csep\rho(x,x)\wedge\varphi(fx)\;.
\]
\item\label{lem:fp-qtop:eff-quots}
$\bftr$ has effective quotients of strong equivalence relations, and up to
postcomposition by isomorphism the regular epimorphisms are precisely the
morphisms of the form $\id:(D,\sigma)\to(D,\tau)$ with
$\tau(x,x)\vdash\sigma(x,x)$.
\item\label{lem:fp-qtop:reg-pullback}
Regular epimorphisms in $\bftr$ are stable under pullback.
\item\label{lem:fp-qtop:pow-objs}
The presheaves
\[
 \subc(-\times (C,\rho)) : (\bftr)^\op\to\catset
\]
are representable.
\item\label{lem:fp-qtop:q-top}
$\bftr$ is a q-topos.
\end{enumerate}
\end{lemma}

\begin{proof}
\emph{Ad \ref{lem:fp-qtop:flims}.} 
Binary products of $(C,\rho)$ and $(D,\sigma)$ are given by $(C\times
D,\rho\perprod\sigma)$, where $\rho\perprod\sigma$ is defined as%
\[
(\rho\perprod\sigma)(c,d,c',d')\defpred
\rho(c,c')\wedge\sigma(d,d').
\]

$(1,\top)$ is a terminal object ($\top$ denotes the greatest predicate in the
fiber over $1\times 1\cong 1$).

An equalizer of $f,g:(C,\rho)\to (D,\sigma)$ is given by
$\id:(C,\tau)\to(C,\rho)$, where $\tau$ is defined as
\[
\tau(x,y)\defpred\rho(x,y)\wedge\sigma(fx,gx).
\]

\medskip
\emph{Ad \ref{lem:fp-qtop:omega}.}
This follows from the transitivity and symmetry of logical equivalence.

\smallskip
\emph{Ad \ref{lem:fp-qtop:epis}.}
We give the proof of this statement in detail to give the reader an idea of how
to carry out this kind of argument in the internal language.
Assume that $e:(C,\rho)\to(D,\sigma)$ in $\funf\trip$ such that
$\sigma(y,y)\vdash\exists x\qdot \rho(x,x)\wedge\sigma(ex,y)$ holds in
$\trip$,
and that $g,h:(D,\sigma)\to(E,\eta)$ such that $ge=he$ as morphisms of
$\funf\trip$, i.e., $\rho(x,x)\fCenter\eta(g(ex),h(ex))$ in $\trip$. We infer
\[
\def\ScoreOverhang{1pt} 
\def\fCenter{ \vdash }
\Axiom$\sigma(y,z)\fCenter\eta(gy,gz)$
\UnaryInf$\sigma(y,ex)\fCenter\eta(gy,g(ex))$
\Axiom$\rho(x,x)\fCenter\eta(g(ex),h(ex))$
\Axiom$\sigma(z,y)\fCenter\eta(hz,hy)$
\UnaryInf$\sigma(ex,y)\fCenter\eta(h(ex),hy)$
\TrinaryInf$\sigma(ex,y),\rho(x,x)\fCenter \eta(gy,hy)$
\DisplayProof.
\]
Here we use as hypotheses that $g$ and $h$ are well defined with respect to
$\sigma, \eta$, and that $ge=he$; then we substitute and reason by transitivity
and symmetry of the relations. 

We proceed by
\[
\Axiom$\sigma(y,y)\fCenter\exists x\qdot \sigma(fx,y)\wedge\rho(x,x)$
\Axiom$\sigma(fx,y),\rho(x,x)\fCenter \eta(gy,hy)$
\UnaryInf$\exists x\qdot \sigma(fx,y)\wedge\rho(x,x)\fCenter \eta(gy,hy)$
\BinaryInf$\sigma(y,y)\fCenter\eta(gy,hy)$
\DisplayProof.
\]
Here we use the remaining assumption and the conclusion of the previous
derivation. The conclusion says that $g=h$ as morphisms of $\funf\trip$, as
desired.

Conversely assume that $e:(C,\rho)\to(D,\sigma)$ is an epimorphism. By the
(semi-)universal property of $\prop$, there exist morphisms $g,h:D\to \prop$ in
$\flc$ such that
\[
 \top \dashv\vdash\triptr(gy)\qtext{and}\exists x\qdot \rho(x,x)\wedge
\sigma(ex,y)\dashv\vdash\triptr(hy).
\]
It is easy to see that these morphisms are well defined with respect to
$\sigma$ and $\pereqv$ and thus give rise to morphisms
$g,h:(D,\sigma)\to(\prop,\pereqv)$ in $\bftr$. Moreover, these morphisms are
equalized by $e$, i.e., $ge=he:(C,\rho)\to(\prop,\pereqv)$, as a calculation in
the logic of $\trip$ shows. Now $e:(C,\rho)\to (D,\sigma)$ is an epimorphism,
and thus we have $g=h:(D,\sigma)\to(\prop,\pereqv)$, which is equivalent to the
validity of the judgment $\sigma(y,y)\vdash \pereqv(gy,hy)$. By unfolding
the definition of $\pereqv$ and making use of the characterizations of $g,h$,
we thus deduce $\sigma(y,y)\vdash\exists x\qdot\rho(x,x)\wedge\sigma(ex,y)$,
as desired.

\smallskip
\emph{Ad \ref{lem:fp-qtop:compat-to-sub}.} It is straightforward to verify that
$\sigma|_\varphi$ is a partial equivalence relation, and that $\id_D$ is well
defined with respect to $\sigma|_\varphi$ and $\sigma$.

To see that $\id:(D,\sigma|_\varphi)\to(D,\sigma)$ is a monomorphism, assume
that $f,g:(C,\rho)\to(D,\sigma|_\varphi)$ such that the compositions with
$\id:(D,\sigma|_\varphi)\to(D,\sigma)$ are equal in $\bftr$, i.e.,
$\rho(x,x)\vdash\sigma(fx,gx)$ in $\trip$. Because
$f:(C,\rho)\to(D,\sigma|_\varphi)$, we also have $\rho(x,x)\vdash\varphi(fx)$,
and the conjunction $\rho(x,x)\vdash\sigma(fx,gx)\wedge\varphi(fx)$ of the two
statements is equivalent to $f=g:(C,\rho)\to(D,\sigma|_\varphi)$.

It remains to show that the map is even a cocover. To see this, consider the
square
\[
\cxymatrix{
(D,\sigma)\depi[d]_e\ar[r]^f & (C,\rho|_\varphi)\ar[d]^\id \\
(E,\eta)\ar[r]_g & (C,\rho)
}
\]
with $e:(D,\sigma)\to(E,\eta)$ an epimorphism.
We have to verify that $g$ is well defined with respect to $\eta$ and
$\rho|_\sigma$, which amounts to establishing $\eta(y,y)\vdash\varphi(gx)$.
This can be derived by applying the characterization of epimorphism of the
previous item to $e$. Finally it is easy to see that the two induced triangles
commute.

\smallskip
\emph{Ad \ref{lem:fp-qtop:mor-to-compat}.}
This is completely straightforward using the previously established facts.

\smallskip
\emph{Ad \ref{lem:fp-qtop:compat-sub-bij}.}
Also straightforward.

\smallskip
\emph{Ad \ref{lem:fp-qtop:reindexing}.}
Also straightforward.

\smallskip
\emph{Ad \ref{lem:fp-qtop:eff-quots}.}
Via the bijection established in \ref{lem:fp-qtop:compat-sub-bij}, the strong
equivalence relations on $(C,\rho)$ correspond to predicates
$\tau\in\trip_{C\times C}$ which are partial equivalence relations, compatible
with $\rho\perprod\rho$, and furthermore total with respect to $\rho$ in the
sense that $\rho(x,x)\vdash\tau(x,x)$. 
It turns out that for partial equivalence relations, the conjunction of
totality and compatibility can be expressed in a slightly simplified manner: A
partial equivalence relation $\tau$ is compatible with respect to
$\rho\perprod\rho$ and total (with respect to $\rho$), iff the judgments
\begin{equation}
\tau(x,x)\vdash\rho(x,x)\qtext{and}\rho(x,y)\vdash\tau(x,y)\label{eq_ter}
\end{equation}
hold in $\trip$. Thus, summing up, the strong equivalence relations on
$(C,\rho)$ correspond to the predicates on $C$ which are partial equivalence
relations \emph{and} satisfy the judgments \eqref{eq_ter}.

Given such a representative $\tau$ of a strong equivalence relation, the
obvious
candidate for the quotient map is
\[
 \id:(C,\rho)\to(C,\tau).
\]
We have to verify that
\[
 (C\times
C,(\rho\perprod\rho)|_\tau)\rightrightarrows(C,\rho)\rightarrow(C,\tau)
\]
is indeed a coequalizer diagram. This is straightforward and left to the reader
(Hint: Show first that the map is an epimorphism using \ref{lem:fp-qtop:epis}).

To verify that the quotient is effective, note that the kernel pair of a map
$f:A\to B$ in a finite limit category can be computed as the pullback of the
diagonal $\delta:B\to B\times B$ along $f\times f$, and then use the
representation of the fibration of cocovers given in
\ref{lem:fp-qtop:reindexing}.

The second part of the claim follows because every regular epimorphism is the
quotient of its kernel pair, which is a strong equivalence relation.

\smallskip
\emph{Ad \ref{lem:fp-qtop:reg-pullback}.}
If we pull a regular epimorphism in `normalized
presentation' $\id:(D,\sigma)\to(D,\tau)$ with
$\tau(x,x)\vdash\sigma(x,x)$  back along a morphism
$f:(C,\rho)\to
(D,\tau)$, we get a square of the form
\[
\cxymatrix{
(C\times D,\theta) \ar[r]^{\pi_1}\ar[d]_{\pi_0} & (D,\sigma)\ar[d]^\id\\
(C,\rho)\ar[r]_f & (D,\tau)
},
\]
where $\theta(c,d,c',d')\idefpred
\rho(c,c')\wedge\sigma(d,d')\wedge\tau(fc,d)$. To see that $(C\times
D,\theta)\to (C,\rho)$ is a regular epimorphism, observe that it is isomorphic
to $(C\times D,\theta)\to (C\times D,\xi)$ with $\xi(c,d,c',d')\idefpred 
\rho(c,c')\wedge\tau(d,d')\wedge\tau(fc,d)$ via the isomorphism $\langle
\id,f\rangle:(C,\rho)\rightleftarrows (C\times D,\xi):\pi_0$, and $(C\times
D,\theta)\to (C\times D,\xi)$ is of the canonical form for regular
epimorphisms.

\smallskip
\emph{Ad \ref{lem:fp-qtop:pow-objs}.}
For a given object $(C,\rho)$ of $\bftr$, we define its power object as
$(\tripower C,\tripower\rho)$ with
\begin{align*}
(\tripower\rho)(m,n)\defpred& (\forall x\qdot x\in m\Rightarrow\rho(x,x))\\
&\wedge (\forall x,y\qdot  x\in m\wedge\rho(x,y)\Rightarrow y\in m) \\
&\wedge (\forall x\qdot x \in m\Leftrightarrow x\in n)
\end{align*}
The cocover corresponding to the element predicate is represented by
\[
 x\in_\rho m \defpred x\in_C m\wedge(\tripower\rho)(m,m)
\]
The verification that these definitions make sense and give rise to a
representation of $\subc(-\times(C,\rho))$ are tedious, but straightforward
(again, the use of \ref{lem:fp-qtop:reindexing} is essential).

\smallskip
\emph{Ad \ref{lem:fp-qtop:q-top}.}
This follows from items \ref{lem:fp-qtop:flims}, \ref{lem:fp-qtop:eff-quots}
and \ref{lem:fp-qtop:pow-objs}.
\end{proof}

\begin{remark}
In Section~\ref{suse_conv}, we stated that we always want to work with
\emph{chosen} limits and colimits whenever we speak about such structures. In
light of this statement, if we asserted in the previous lemma that $\funf\trip$
always \emph{has} finite limits and quotients of strong equivalence relations,
from now on we think of it as equipped with the choices of limits and quotients
whose constructions were sketched in the proof. On the contrary, we are happy
with the mere \emph{existence} of power objects.
\end{remark}

After having described the object part of $\funf$, we come to its action on
1-cells.
Assume that $\effi:\trip\to\triq$ is a tripos morphism. The functor
$\funf\effi:\funf\trip\to\funf\triq$ is given by
\[
\vcenter{
 \xymatrix@!0{
(C,\rho)\ar[dd]^f & & \mapsto & & (FC,\Phi_{C,C}\rho)\ar[dd]^{Ff}\\
&&\mapsto\\
(D,\sigma) & & \mapsto & & (FD,\Phi_{D,D}\sigma)\\
}}
.
\]
See Definition~\ref{def_coherence-stuff}.\ref{def_coherence-stuff_relsym} for
the notation $\Phi_{C,C}\rho$.
To see that this construction is well defined, we have to show that
$\Phi_{C,C}\rho, \Phi_{D,D}\sigma$ are partial equivalence relations and that
$f\mapsto Ff$ is compatible with the partial equivalence relations on
$\flc(C,D)$ and $\fld(FC,FD)$. This is a consequence of
Corollary~\ref{cor_trimo-judgement}. Functoriality is clear, and furthermore
we have:
\begin{lemma}\label{lem_efeffi}
For every tripos morphism $\effi:\trip\to\triq$, $\funf\effi$ preserves finite
limits (thus in particular cocovers), and covers. 

If $\Phi$ commutes with existential quantification \emph{along projections},
then $\funf\effi$ also maps epimorphisms to epimorphisms.
\end{lemma}
\begin{proof}
These claims follow from Corollary~\ref{cor_trimo-judgement}, using the
construction of finite limits in the proof of
Lemma~\ref{lem:fp-qtop}.\ref{lem:fp-qtop:flims}, the characterization of
regular epimorphisms in 
\ref{lem:fp-qtop}.\ref{lem:fp-qtop:eff-quots}, and the characterization of
epimorphisms given in \ref{lem:fp-qtop}.\ref{lem:fp-qtop:epis}.
\end{proof}

Now we come to the action of $\funf$ on 2-cells. Let
$\eta:\effi\to\ggamma:\trip\to\triq$ be a tripos transformation. Then we can
define a natural transformation $\funf\eta:\funf\trip\to\funf\triq$ whose
component at $(C,\rho)$ is
\[
 \eta_C: (FC,\Phi\rho)\to (GC,\Gamma\rho)
\]
The verification that this make sense is straightforward and left to the
reader.

Now the description of $\funf$ is complete, and we can show:
\begin{lemma}\label{lem_f_spec_strict}
The previous constructions establish a \emph{strict special functor}
\[
 \funf:\cattrip\to\catqtop.
\]
\end{lemma}
\begin{proof}
Well definedness follows from Lemmas~\ref{lem:fp-qtop}.\ref{lem:fp-qtop:q-top}
and \ref{lem_efeffi}.
By strict, we mean that the construction is 2-functorial. This is
straightforward to verify; note that here it is important that the fibers of
the triposes are posets and not mere preorders.
Finally, $\funf$ is special because of
Lemma~\ref{lem_efeffi}.
\end{proof}

\subsubsection{The unit \texorpdfstring{$\pseta$}{eta}}\label{sec_f_s_unit}

The unit of $\funf\dashv\funs$ is a special transformation
$\pseta:\id_\cattrip\to \funs\funf$. Its component at $\trip:\flx\to\flc$ is
the tripos transformation
\[\pseta_\trip = (D_\trip,\Delta_\trip):\trip\to\funs\funf\trip,\]
which we describe now.

$D_\trip:\bbc\to\funf\trip$ is defined by
\[
 \vcenter{\xymatrix@!0{
C\ar[dd]^f & & \mapsto & & (C,=)\ar[dd]^{f}\\
&&\mapsto\\
D & & \mapsto & & (D,=)\\
}}.
\]
Using the same techniques as in the previous section, it is quite easy to see
that $D_\trip$ is a finite product preserving functor. \emph{(It does, however,
in general not preserve equalizers!)}

$\Delta_\trip:\flx\to\funcoc(\funt\trip)$ maps predicates $\varphi\in\trip_C$
to subobjects
\[
 (C,\predeq|_\varphi)\ecocov (C,\predeq).
\]
It is not difficult to verify that this defines a fibered functor over
$D_\trip$, and that $(D_\trip,\Delta_\trip)$ is a regular tripos morphism.

For a tripos transformation $\effi:\trip\to\triq$, the transformation
constraint 
\[
\vcenter{\xymatrix@!@+0.5pc{
\trip
	\ar[r]^{\effi}
	\ar[d]_{\eta_\triq} &
\triq
	\ar[d]^{\eta_\trip}
\\
\funs\funf\trip
	\ar[r]_{\funs\funf\effi} &
\funs\funf\triq
\ar@{<=}@<-1ex>"2,1"*+++++\frm{};"1,2"*++++\frm{}^{\eta_{\effi}\!\!\!}
}} \quad= \sdi{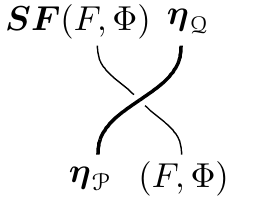}
\]
has components
\begin{equation}
\eta_{\effi,C} = \id_{FC}: (FC,=)\to (FC,\Phi\predeq).
\label{eq_fs_unit_constr}
\end{equation}
The verifications that $\eta_{\effi}$ is a natural transformation that gives
rise to a tripos transformation, and that $\eta$ is a lax transformation are
left to the reader. To verify that $\eta$ is even a special transformation, we
have to show that $\eta_{\effi}$ is invertible whenever $\effi$ commutes with
existential quantification. This is apparent from the presentation
\eqref{eq_fs_unit_constr} of the components of $\eta_\effi$, because the
equality predicate is defined in terms of existential quantification, and thus
$(\Phi\predeq)=(\predeq)$.

\subsubsection{The counit
\texorpdfstring{$\psve$}{epsilon}}\label{suse_f-s-ve}

First of all a linguistic remark. In this section, we consider the
category $\fsc$, which is formally obtained by applying a tripos theoretic
construction, defined in the internal language of triposes,  to the fibration
of cocovers on a q-topos. This leads us to reasoning in $\funs\qcc$
using not the \emph{core calculus} of Section~\ref{sec_q_toposes} but the
higher order logic of Section~\ref{sec_triposes}. This doesn't matter too much
since higher order intuitionistic can be embedded into the core calculus
anyway, but has the formal consequence that predicates will not be
$\Omega$-valued morphisms as in Section~\ref{sec_q_toposes}, but equivalence
classes of cocovers.

\medskip

The components of the counit
$\psve:\funf\funs\to\id_\catqtop$ of $\funf\dashv\funs$ are functors 
\[\psve_\qcc:\fsc\to\qcc,\]
which are defined as follows. An object of $\funf\funs\qcc$ is a pair
$(C,[r])$,
where $C$ is an object of $\qcc$ and $[r]$ is an equivalence
class of cocovers $r:R\ecocov C\times C$ which is a partial
equivalence relation in $\funs\qcc$. Let $m:C_0\ecocov C$ be a representative%
\footnote{The constructions here rely on the presence of chosen limits and
colimits, and depend on the choice or representatives of predicates, for
example for the construction of certain pullbacks. Thus, it seems that our
decision to quotient out the subobject fibrations forces us here to use
choice. This is not true, however, since we can always obtain a canonical
representative of a subobject as $\chi^*\toptr$, where $\chi$ is the
characteristic function of the subobject.} of
the predicate $\ilbracks{x\csep r(x,x)}$. Then the predicate $[r_0](x,y)
\idefpred
[r](mx,my)$ is an equivalence relation on $C_0$, and we define
$C/r$ to be its quotient in $\qcc$. We thus have a
subquotient
\[ C/r \stackrel{e}{\ecovleft} C_0 \stackrel{m}{\ecocov}C.\]
Given another object $(D,[s])$ and a morphism $f:(C,[r])\to(D,[s])$
in $\funf\funs\qcc$, we claim that there is
a unique pair of mediators for the diagram
\begin{equation}
 \cxymatrix{
C/r\dashed[d]_{h} &
C_0\cov[l]_e\dashed[d]^g\cocov[r]^m & C\ar[d]^f \\
D/s & D_0\cov[l]^{e'}\cocov[r]_{m'} & D
}.\label{eq_def_ecf}
\end{equation}
Uniqueness is clear, since $m'$ is mono and $e$ is epic, and the existence of
$g$ and $h$ follows from the validity of $[r](x,x)\vdash [s](fx,fx)$ and
$[r_0](x,y)\vdash [s_0](gx,gy)$, respectively. 

We define
\[\psve_\qcc(C,[r])=C/r,\quad\psve_\qcc(D,[s])=D/s, \qtext{and}
\psve_\qcc(f)=h.\]
Functoriality of $\psve_\qcc$ follows from universal properties. Furthermore,
we have the following:

\begin{lemma}
The functor $\psve_\qcc:\fsc\to\qcc$ preserves finite limits, epimorphisms and
regular epimorphisms. This means that it is a \regular\ 1-cell in the
pre-equipment $\catqtop$.
\end{lemma}
\begin{proof}
The terminal object is clearly preserved, since it is given by $(1,[\id])$ in
$\fsc$. 

For products, remember that a product of $(C,[r])$ and $(D,[s])$ is given by
$(C\times D,[r]\perprod [s])$. We form the subquotient spans
$Q\stackrel{e}{\ecovleft} C_0 \stackrel{m}{\ecocov} C$ and
$R\stackrel{p}{\ecovleft} D_0 \stackrel{n}{\ecocov} D$
corresponding to the strong equivalence relations $[r]$ and $[s]$. We want to
show
that
$Q\times R \stackrel{e\times p}{\ecovleft} C_0\times D_0 \stackrel{m\times
n}{\ecocov} C\times D$ is a subquotient span for $[r]\perprod [s]$. First of
all,
we
note that the classes of cocovers and regular epimorphisms are both closed
under
pullback and composition, and thus under products%
\footnote{
The regular epimorphisms are closed under composition because they coincide
with the
covers.
}.
This implies that the legs of the product span are again a regular epi and a
cover.

To see that $[m\times n](x,y)=[r](x,x)\wedge [s](y,y)$, observe that every
predicate $[m]$ is equivalent to $\ilbracks{y\csep
\exists x\qdot [m]x=y}$. 
Thus $[m\times n]$ coincides with $\ilbracks{x,y\csep \exists
x_0,y_0\qdot mx_0=x\wedge ny_0=y}$, which is equivalent to $\ilbracks{x,y\csep
(\exists
x_0\qdot mx_0=x)\wedge(\exists y_0\qdot ny_0=y)}$ and to $\ilbracks{x,y\csep
r(x,x)\wedge s(y,y)}$.

Next we want to show that $e\times p$ is a coequalizer of the components of
$[r_0]\perprod [s_0] = \ilbracks{x,y,x',y'\csep [r](mx,mx')\wedge
[s](ny,ny')}$. Because
$e\times p$ is a regular epimorphism, it suffices to show that its kernel pair
is equivalent to $[r_0]\perprod [s_0]$. This follows immediately, because the
kernel pair can be expressed in the internal language as
$\ilbracks{x,y,x',y'\csep ex=ex'\wedge py=py'}$.

A similar argument shows the preservation of equalizers.

\smallskip

Now let $f:(C,[r])\to (D,[s])$ be an epimorphism in $\fsc$. To show that its
image
under
$\psve_\qcc$ --- i.e.\ the arrow $h$ in diagram \eqref{eq_def_ecf} is again an
epimorphism, we have to show that $\ilbracks{v\csep\vdash\exists u\qdot fu=v}$
holds,
but this
can be derived from the valid judgments
$\ilbracks{x\csep [r](x,x)\vdash \exists x_0\qdot mx_0=x}$, $\ilbracks{u\csep
\exists x_0\qdot
ex_0=u}$, $\ilbracks{y_0\csep\vdash [s](ny_0,ny_0)}$ and $\ilbracks{y\csep
[s](y,y)\vdash\exists
x\qdot [r](x,x)\wedge fx=y}$.

\smallskip

Finally, we have to show that $\psve_\qcc$ preserves regular epimorphisms. For
a
regular epimorphism in canonical representation, i.e. $\id:(C,[r])\to(C,[s])$
with
${[s](x,x)\vdash [r](x,x)}$, the diagram of subquotient spans looks as
follows
\[
 \cxymatrix{
C/r\ar[d]_{h} & C_0\cov[l]_e\ar[d]^\id\cocov[r] & C\ar[d]^\id \\
C/s & C_0\cov[l]^{e'}\cocov[r] & C
},
\]
and the claim follows from the facts that in a composition the second arrow is
a cover whenever the composition is a cover (follows from orthogonality), and
covers coincide with regular epimorphisms (Lemma~\ref{lem_q-topos-regular}).
\end{proof}

Next we define the transformation constraint
\[
\vcenter{\xymatrix@!{
\fsc\ar[r]^{\funf\funs F}\ar[d]_{\psve_\qcc} & \fsd\ar[d]^{\psve_\qcd} \\
\qcc\ar[r]_F & \qcd
\ar@{<=}@<-.5ex>"2,1"*+++++\frm{};"1,2"*+++++\frm{}^{\psve_{F}}
}}
\]
The subquotient span $C/r \stackrel{e}{\ecovleft} C_0 \stackrel{m}{\ecocov} C$
associated to the image of an object $(C,[r])\in\objects{\fsc}$ under
$\psve_\qcc$ gets mapped to 
\begin{equation}
F(C/r) \stackrel{Fe}{\leftarrow} FC_0 \stackrel{Fm}{\ecocov}
FC\label{eq:f-subquot}
\end{equation}
by $F$. $Fe$ is not necessarily an epimorphism any more, but because $F$
preserves finite limits, $F([r]|_{[m]})$ is still its kernel.

On the other side of the diagram, $(C,[r])$ gets mapped first to $(FC,[Fr])$ by
$\funf\funs F$, and then to the object $C'_0/Fr$ in the subquotient span
\begin{equation}
FC/Fr \stackrel{m'}{\ecovleft} C_0' \stackrel{e'}{\ecocov} FC
\label{eq:subquot-f}
\end{equation}
by $\psve_\qcd$.
Now the support $C_0'$ of $Fr$ is  isomorphic to the image $FC_0$ of the
support of $m$ under $F$, because the support is defined as a pullback, and
those are preserved by $F$.
If we combine \eqref{eq:f-subquot} and \eqref{eq:subquot-f} into a diagram
\begin{equation}
 \cxymatrix{
FC/Fr\ar@{ >-->}[d]  & C_0'\cocov[r]^{m'}\ar[d]_\cong\cov[l]_{e'} & FC \\
F(C/r) & FC_0\cocov[ur]_{Fm}\ar[l]^{Fe}
},\label{eq:combi-dia}
\end{equation}
we see that the universal property of the quotient map $e'$ allows us to
construct a mediating arrow from $FC/Fr$ to $F(C/r)$ (which is moreover monic
since $e'$ and $Fe$ have isomorphic kernel pairs). This map is the component
$\psve_{F,(C,r)}$ of the transformation constraint $\psve_F$. To prove that
$\psve$ is a well defined oplax transformation, we have to verify naturality of
$\psve_F$ and the transformation axioms for $\psve$, which is straightforward; 
all the commutations follow from universal properties. 

To verify that $\psve$ is moreover a special transformation, we have to verify
that $\psve_F$ is invertible whenever $F$ is \regular. This becomes clear when
looking at diagram~\eqref{eq:combi-dia}. When $F$ is \regular\ then $Fe$ is a
regular epimorphism, because those are preserved by \regular\ functors between
q-toposes by definition. We then have two regular epimorphisms $e'$, $Fe$ with
the same kernel pair and thus the mediating arrow is an isomorphism.

\subsubsection{The modification \texorpdfstring{$\nu$}{nu}}

The components of the modification
\[
\nu: 
\psve_{\funf}\circ \funf\pseta 
\to 
\id_{\funf}
: \funf\to \funf :\cattrip\to\catqtop
\]
are natural transformations
\[
\nu_\trip:
\psve_\bftr\circ\funf\pseta_\trip
\to 
\id_\bftr
:\bftr\to\bftr,
\]
whose components are in turn morphisms of type
\[
\nu_{\trip,(C,\rho)}:
\psve_\bftr(\funf\pseta_\trip\cro)
\to 
(C,\rho)
:\bftr.
\]
Now $\funf\pseta_\trip\cro$ is the object $((C,\predeq),[\tilde{\rho}])$ in
$\funf\funs\bftr$, where
\[\tilde{\rho}: (C\times C,(\predeq\perprod\predeq)|_\rho)\ecocov(C\times
C,(\predeq\perprod\predeq))\]
is the subobject induced by $\rho$.

General objects of $\fsfp$ are of the form $(\cro,[r])$ where
$r:U\ecocov(C\times C,\rho\perprod\rho)$ 
is a partial equivalence relation on $\cro$ which can be represented by a
predicate $\tau\in\trip_{C\times C}$.

We know from the proof of Lemma~\ref{lem:fp-qtop}.\ref{lem:fp-qtop:eff-quots}
and from Section~\ref{suse_f-s-ve} how to construct a subquotient span for the
subquotient of  $(C,\rho)$ by $r$.
It is given by
\[
(C,\tau)\ecovleft(C,\rho|_{\supp(\tau)})\ecocov(C,\rho).
\]
We pointed out earlier that we have to be careful not to mix up chosen limits
and colimits with arbitrary limits and colimits, so at this point we remark
that $(C,\tau)$ is not precisely the image of $((C,\rho),[r])$ under
$\psve_\bftr$, because the object in the middle of the subquotient span (the
support of the partial equivalence relation) is defined by a pullback in
$\funf\trip$ and if we carry out the canonical pullback construction as an
equalizer of a product, and use the choices of products and equalizers
described in the proof of Lemma~\ref{lem:fp-qtop}, we end up with a partial
equivalence relation on $C\times C\times C$. Thus, all we can say at this point
is that there is a canonical isomorphism $\psve_\bftr((C,\rho),r)\cong
(C,\tau)$ which we obtain by comparing the subquotient span above with the one
arising from the canonical constructions.

Instantiating with $\funf\pseta_\trip\cro=((C,\predeq),[\tilde{\rho}])$, we
obtain the desired isomorphism
\[
\nu_{\trip,(C,\rho)}:\psve_\bftr(\pseta_\trip\cro)\stackrel{\cong}{\to}(C,
\rho).
\]

The reader is invited to verify that $\nu_\trip$ is indeed a natural
transformation and that $\nu$ is a modification.

\subsubsection{The modification \texorpdfstring{$\mu$}{mu}}
 
The modification $\mu$ has the type
\[
\mu \wcol 
\id_\funs\to\funs\psve\circ\pseta_\funs
\wcol\funs\to\funs \wcol 
\catqtop\to\cattrip.
\]
Its components are tripos transformations
\[
 \mu_\qcc\wcol
\id_\bsqc\to\funs\psve_\qcc\circ\pseta_\bsqc
\wcol\bsqc\to\bsqc
\]
whose components  are morphisms
\[
\mu_{\qcc,C}\wcol
C\to\psve_\qcc(D_\bsqc C)
\]
in $\qcc$. Applying the definition in the previous sections, we see that
$\psve_\qcc(D_\bsqc C)=C/\delta_C$ in the (degenerate) subquotient span
\[
C/\delta_C{}\stackrel{\cong}{\leftarrow}\bullet\stackrel{\cong}{\to} C,
\]
i.e.\ the (sub)quotient of $C$ by the diagonal predicate. Of course a possible
subquotient span of $C$ by $\delta_C$ would be the span consisting of two
identities, but we can not use this since we pledged earlier always to use
instances of limits and (co)limits that are given by some construction of
chosen (co)limits which is given to us together with the definition of $\qcc$.
We now define $\mu_{\qcc,C}$ in the most straightforward way --- as composition
of the two isomorphisms in the subquotient span.

It remains to check that this defines a natural transformation
$\mu_\qcc:\id_\qcc\to\psve_\qcc\circ D_{\funs\qcc}$ which induces a tripos
transformation $\mu_\qcc:\id_\bsqc\to\funs\psve_\qcc\circ\pseta_\bsqc$, and
that these tripos transformations give rise to a modification $\mu
:\id_\funs\to\funs\psve\circ\pseta_\funs$. The naturality of $\mu_\qcc$ follows
directly from the definition of $\psve_\qcc$, and the fact that $\mu_\qcc$ is a
tripos transformation is also easy to see. The verification of the modification
axiom boils down to showing that for any finite limit preserving functor
$F:\qcc\to\qcd$ between q-toposes and any $C\in\objects{\qcc}$, the square
\[
 \xymatrix@C+20pt{
FC/\delta_C\ar[r]^-{\psve_\qcd(\pseta_{\funs F,C})} & FC/F\delta_C
\ar[d]^-{\psve_{F,(C,[\delta])}}\\
FC\ar[u]^-{\mu_{\qcd,FC}}\ar[r]_{F\mu_{\qcc,C}} & F(C/\delta_C)
}
\]
commutes. A careful inspection of the constructions shows that all the arrows
arise as mediating arrows between spans which are anchored at $FC$ on one side
(and furthermore have isomorphic legs), and from this commutation is evident.

\subsubsection{The axioms}

To check that the data $F,S,\pseta,\psve,\mu,\nu$ form a special biadjunction,
we have to check the equalities of modifications stated in
\eqref{eq_spec_adj_ax}. Equality of modifications is componentwise equality,
and for the components, which are tripos transformations for the first axiom,
equality is in turn componentwise equality of the underlying natural
transformations. By evaluating at $\trip:\flx\to\flc$ and $C\in\objects{\flc}$,
we see that we have to check the commutativity of the following square.
\[
 \xymatrix@1@C+40pt{
\psve_{\funf\trip}(D_{\funs\funf\trip}(D_\trip C))
    \ar[r]^{\psve_{\funf\trip}(\pseta_{\pseta_\trip,C})}
& 
\psve_{\funf\trip}(\funf\pseta_{\trip}(D_\trip C))
    \ar[d]^{\nu_{\trip,D_\trip C}}
\\
D_\trip C
    \ar[u]^{\mu_{\funf\trip,D_\trip C}}
    \ar[r]_\id
& 
D_\trip C
}
\]
Instantiating the constructions, we obtain
\[
\vcenter{
 \xymatrix@1@C+40pt{
(C,\predeq)/\predeq 
    \ar[r]^{\psve_\bftr(\pseta_{\pseta_\trip,C})}
& 
(C,\predeq)/\Delta\predeq 
    \ar[d]^{\nu_{\trip,(C,\predeq)}} 
\\
(C,\predeq)
    \ar[u]^{\mu_{\funf\trip,(C,\predeq)}}
    \ar[r]_\id
& 
(C,\predeq)
}}.
\]
To see that this commutes, observe that the underlying arrows in $\flc$ are all
equal to $\id_C$, regardless of the partial equivalence relations.

\smallskip

The verification of the second axiom amounts to checking that for every q-topos
$\qcc$ and every $(C,[r])\in{\funf\funs\qcc}$, the square
\[
 \xymatrix@C+40pt{
\psve_\qcc(\funf\funs\psve_\qcc(\funf\pseta_{\funs\qcc}(C,[r])))
    \ar[r]^{\psve_{\psve_\qcc}(\funf\pseta_{\funs\qcc}(C,[r]))}
&
\psve_\qcc(\psve_{\funf\funs\qcc}(\funf\pseta_{\funs\qcc}(C,[r])))
    \ar[d]^{\psve_\qcc(\nu_{\funs\qcc}{\funs\qcc}(C,[r]))}
\\
\psve_\qcc(C,[r])
    \ar[r]_\id
    \ar[u]^{\psve_\qcc(\funf\mu_\qcc(C,[r]))}
&
\psve_\qcc(C,[r])
}
\]
commutes in $\qcc$. The verification of this is cumbersome, but in the end it
boils down to observing that different mediating arrows between spans anchored
at $C$ are the same, which is a similar argument to the one we sketched to
establish that $\mu$ is a modification.

\subsection{The biadjunction \texorpdfstring{$\funt\dashv \funu$}{T,U}
between \\
q-toposes and toposes}\label{suse:adj-t-u}

Now we come to the special biadjunction
$(\funt\dashv\funu,\pseta,\psve,\nu,\mu)$ between the pre-equipments of
q-toposes and toposes 
Fortunately, this is a lot easier
and less technical than the biadjunction $\funf\dashv\funs$ of the previous
section.

\subsubsection{The special functors \texorpdfstring{$\funt$}{T} and
\texorpdfstring{$\funu$}{U}}

The functor
\[
 \funu:\cattop\to\catqtop
\]
is just the inclusion. It is well defined because every topos is a q-topos and
in a topos epimorphisms and regular epimorphisms coincide and therefore any
functor that preserves the former automatically preserves the latter.

The object part of the functor
\[
 \funt:\catqtop\to\cattop
\]
was already given in Lemma~\ref{lem:coarse-objects}; $\funt\qcc$ is the full
subcategory of $\qcc$ on the coarse objects. 

The action of $\funt$ on functors and natural transformations is given by
\begin{align*}
\funt F & = J_\qcd\; F\; I_\qcc & \qtext{where} F & : \qcc\to\qcd \\
\funt \theta & = J_\qcd\; \theta\; I_\qcc & \qtext{where} \theta &: F\to G :
\qcc\to\qcd.
\end{align*}
It follows from
Lemma~\ref{lem:coarse-objects}.\ref{lem_coarse-objects_j-regular} that $\funt
F$ preserves finite limits, and from 
\ref{lem:coarse-objects}.\ref{lem_coarse-objects_i-preserves-epis} that $\funt
F$ is \regular\  whenever $F$ is \regular.

For $F:\qcb\to\qcc$ and $G:\qcc\to\qcd$, the composition constraint
$\funt(GF)\to \funt G\; \funt F$ is given by
\[
 \sdi{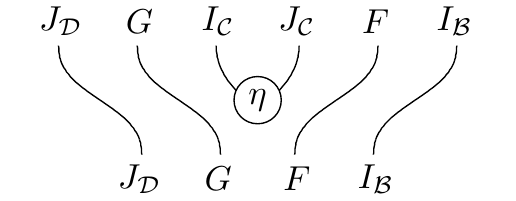},
\]
and the identity constraint $\funt\id_\qcc\to\id_\qcc$ is given by
\[
 \sdi{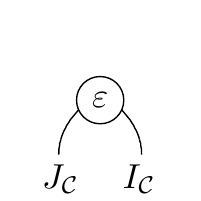}.
\]

The axioms for oplax functors follow formally from 2-categorical reasoning in
$\catcat$, using only the axioms for the adjunctions $J_\qcc\dashv I_\qcc$.

To show that $\funt$ is in fact a special functor, it remains to show that
$\funt(GF)\to\funt G\; \funt F$ is invertible whenever $G$ is \regular. This is
a direct consequence of the following lemma. 
\begin{lemma}\label{lem:nat-inv}
The natural transformation
\[
\sdi{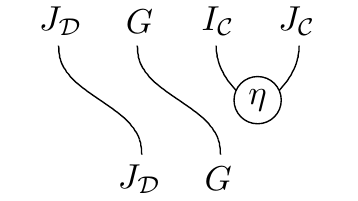}
\]
is invertible whenever $G$ \regular.
\end{lemma}
\begin{proof}
Let $C\in\objects{\qcc}$.
The component at $C$ of the transformation in question is the mediating arrow
in the diagram
\[
 \cxymatrix{
GC\mono[r]^{G\eta_C}\monepi[d] & G\crefl{C}\monepi[d] \\
\crefl{GC}\dashed[r] & \crefl{G\crefl{C}}
}
\]
Now if $G$ is \regular, then $G\eta_C$ is also an epimorphism, and by
coarseness of $\crefl{GC}$, there exists a mediating arrow in the opposite
direction, which turns out to be inverse (this follows because the other arrows
are epimorphisms).
\end{proof}

\subsubsection{The special transformations \texorpdfstring{$\pseta$}{eta} and
\texorpdfstring{$\psve$}{epsilon}}

The unit $\pseta:\id_\catqtop\to \funu\funt$ of $\funt\dashv\funu$ is quite easy
to define. Its components are just the
reflection functors $J_\qcc$. The definition of special transformation requires
them to be \regular, and indeed they are, by
Lemma~\ref{lem:coarse-objects}.\ref{lem_coarse-objects_j-regular}.

The transformation constraint of $\pseta$ at $F:\qcc\to\qcd$ is given by
\[
 \pseta_F = \sdi{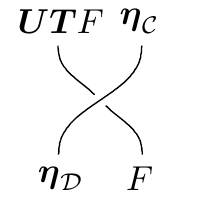} = \sdi{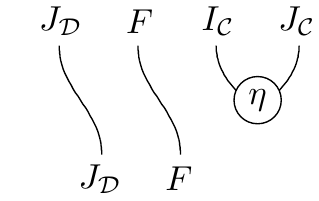}
\]

where $\eta$ is the unit of $J_\qcc\dashv I_\qcc$. 
It follows from general abstract nonsense that this does indeed define an oplax
transformation, which is furthermore special as follows from
Lemma~\ref{lem:nat-inv}.

Now we come to the counit. First of all, a little stylistic/philosophical
remark to explain a design decision. Previously, we assumed that for every
q-topos $\qcc$ we have a chosen reflection $J_\qcc\dashv
I_\qcc:\funt\qcc\to\qcc$ into the subtopos of coarse objects. Now if $\qcc$ is
already a topos, then $\funt\qcc=\qcc$ and we could choose
$I_\qcc=J_\qcc=\id_\qcc$.
However, such a choice depending on whether a q-topos is already a topos is
horribly nonconstructive and not really necessary. Thus, we are content with
the
fact that whenever the q-topos is already a topos, we have an \emph{adjoint
auto-equivalence} $J_\qcc\dashv I_\qcc$.

We define the component of the counit of $\funt\dashv\funu$ at a topos
$\tope$ 
to be 
\[
d\psve_\tope =  I_{\funu\tope} : \funt\funu\tope\to\tope,
\]
which is \regular\ because it is an equivalence.

The transformation constraint
of $\psve$ at $F:\tope\to\topf$ is given by 
\[
\psve_F = \sdi{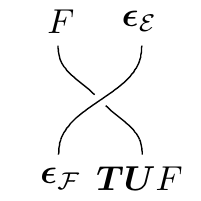} =\sdi{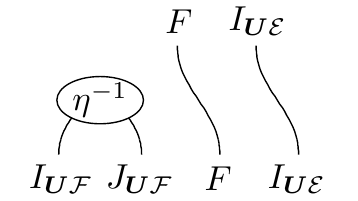}
\]
where $\eta$ is the unit of $J_\topf\dashv I_\topf$. The reader may verify that
this gives a strong transformation.

\subsubsection{The modifications \texorpdfstring{$\nu$}{nu} and
\texorpdfstring{$\mu$}{mu}}

The component $\nu_\qcc$ of the modification
\[
 \nu\qcol \psve\funt\circ\funt\pseta\qarr\id_\funt
\]
is a isomorphic 2-cell in the pre-equipment $\catqtop$ and is given by
\[
\sdi{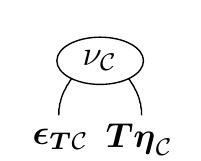} = 
\sdi{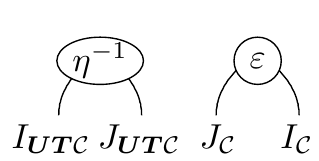},
\]
where $\eta$ is invertible because $J_\utc\dashv I_\utc$ $\ve$ is an
equivalence and is invertible because $J_\qcc\dashv I_\qcc$ is a reflection.

We want to verify that this does indeed define a modification. If we
instantiate the modification axiom
\begin{equation}\label{eq:t-u-nu-ax-generic}
 \sdi{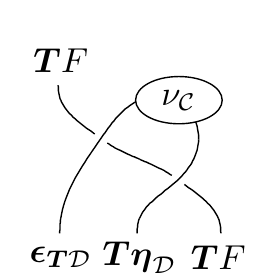} = \sdi{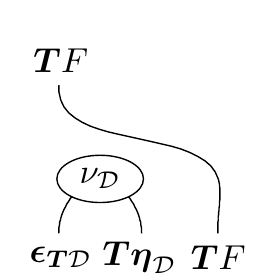}
\end{equation}
with the definitions, we obtain\footnote{The subtle point is how to unfold the
right crossing in the left diagram of \eqref{eq:t-u-nu-ax-generic}. 
This is the constraint cell $\funt\pseta_F$, which is obtained by applying
$\funt$ to the constraint cell $\pseta_F$ and then pre- and postcomposing with
composition constraints of $\funt$.
}
\[
  \sdi{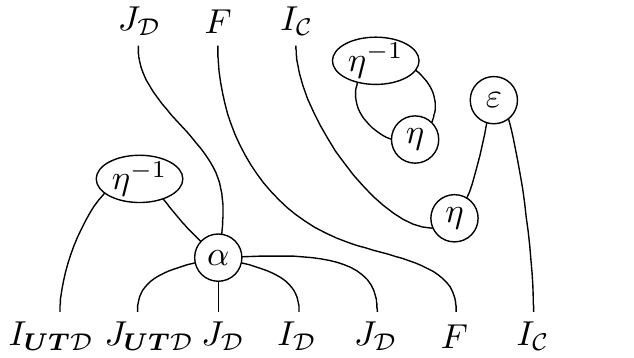} \hspace{-4mm}= \hspace{-3mm}
\sdi{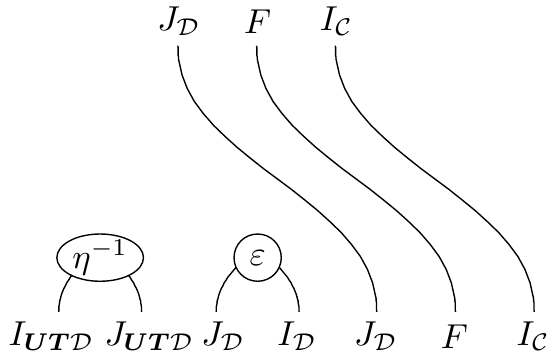},
\]
where $\alpha$ is the inverse of
\[
\sdi{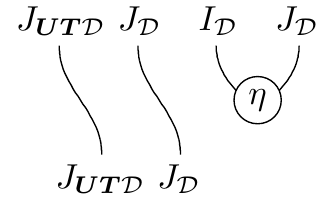},
\]
which is invertible by Lemma~\ref{lem:nat-inv}. The equality now follows
formally from the triangle equalities for the adjunctions $J\dashv I$.

The component of the modification
\[
 \mu\quad:\quad \id_U\qarr\funu\psve\circ\pseta\funu
\]
at $\tope\in\objects{\cattop}$ is given by
\[
\sdi{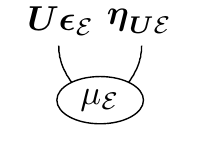} = \sdi{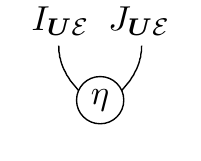}.
\]
The modification axiom for $\mu$ is verified as follows (the first and third
equalities just instantiate definitions):
\[
 \sdi{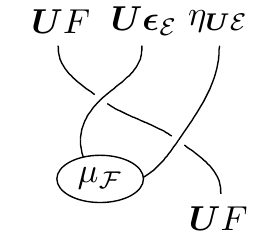}\!\!\!\!\!\!\! = \!\!\sdi{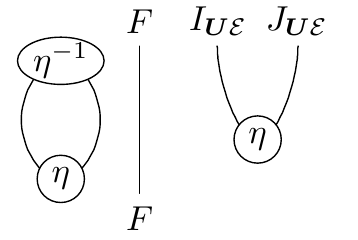}\!\!\!\!\!
=
\!
\sdi{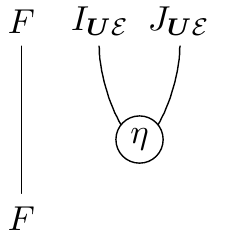}\!\!\! =\!\!\!\! \sdi{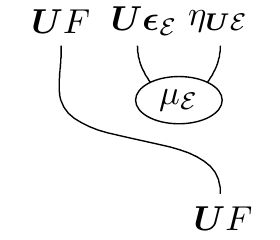}
\]

\subsubsection{The axioms}

Finally we have to verify the biadjunction axioms.
The calculation for the first axiom looks as follows:
\[
 \sdi{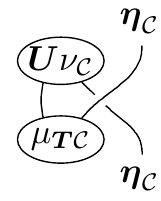}=\sdi{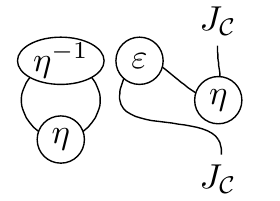}=\sdi{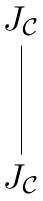}=
\sdi{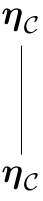}
\]
And here is the calculation for the second axiom:
\[
 \sdi{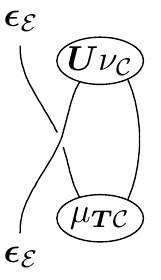}=\sdi{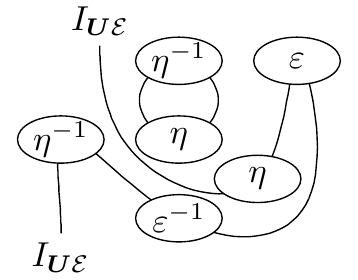}=\sdi{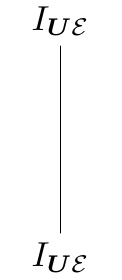}=
\sdi{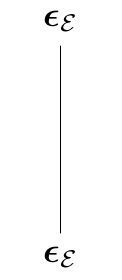}
\]
This finishes the proof that we have a special biadjunction $\funt\dashv\funu$.
We can
compose with the special biadjunction $\funf\dashv\funs$ of
Section~\ref{suse:adj-f-s} to obtain the tripos-to-topos construction, i.e., a
left biadjoint $\funt\circ\funf$ of the forgetful functor $\funs\circ\funu$
from toposes to triposes.

\subsection{Summary, examples}

After a lot of in-detail constructions and verifications, we try to give a
kind of big picture. We have related three \tworeg{} pre-equipments by a pair
of special biadjunctions.
\[
\vcenter{
\xymatrix{
\cattrip\adjr{\funf}{\funs} &
\catqtop\adjr{\funt}{\funu} & \cattop
}}
\]
As right special biadjoints, $\funs$ and $\funt$ are necessarily strong by
Lemma~\ref{lem_right_strong}, and from Lemma~\ref{lem_f_spec_strict} we know
that $F$ is strong as well (in fact all three functors are even strict). The
counit of $\funt\adj\funu$ is also strong, but the remaining constructs
are genuinely oplax, as we will exemplify in the following.

\medskip

That $\funt$ is necessarily oplax becomes apparent already from the example in
the introduction (if $\funt$ were strong as well as $\funf$, then so would be
their composite). 

\medskip

To show that the counit of $\funf\adj\funs$ is not a
strong transformation, we have to give a finite limit preserving functor
$F:\qcc\to\qcd$ between q-toposes such that the transformation constraint
$\psve_F$ is not
invertible. 

Consider the poset $D=\{\bot,l,r\}$ with $\bot\leq l$ and
$\bot\leq r$, and the global sections functor $\Gamma:\hat{D}\to\catset$. We
define a contravariant presheaf $A$ on $D$ by $A(\bot)=\{0,1\}$, $A(l)=\{0\}$,
and $A(r)=\{1\}$ where the restriction maps are just the inclusions, and define
$R$ to be the maximal equivalence relation
$R\emono A\times A$ on $A$. Then $(A,R)$ is an object of $\funf\funs\hat{D}$,
and we follow it along the two sides of the square
\[
\vcenter{\xymatrix@C+5mm{
\funf\funs(\hat{D})
    \ar[r]^{\funf\funs(\Gamma)}
    \ar[d]_{\psve_{\hat{D}}}& 
\funf\funs(\catset)
    \ar[d]^{\psve_\catset}\\
\hat{D}
    \ar[r]_\Gamma & 
\catset
}}.
\]
We have $\psve_\catset(\funf\funs(\Gamma)(A,R))=\varnothing$, since $A$
doesn't have any global elements, but $\Gamma(\psve_{\hat{D}}(A,R))=1$, since
$A$ has global support, thus $\psve_\Gamma$ is not a natural isomorphism.

\medskip

Finally, we will have a closer look at the units of $\funf\adj\funs$
and $\funt\adj\funu$, as a demonstration how our two-step decomposition gives
additional information in form of an epi-mono factorization.

Consider a triposes $\trip:\flx\to\flc$ and $\triq:\fly\to\fld$ and a tripos
 transformation $\effi:\trip\to\triq$. the unit constraint cell of
the composite biadjunction $\funt\funf\adj\funs\funu$ is a natural
transformation which is given by the pasting 
\begin{equation}\label{eq_unit_decomp}
\vcenter{\xymatrix@C+1.1cm{
\flc
    \ar[r]^F
    \ar[d]_{D_\trip}
    \emar[rd]|{\pseta_\effi\Downarrow} &
\fld
    \ar[d]^{D_\triq} \\
\funf\trip
    \ar[r]|{\funf\effi}
    \ar[d]_{J_{\funf\trip}}
    \emar[rd]|{\pseta_{\funf\effi}\Downarrow} &
\funf\triq
    \ar[d]^{J_{\funf\triq}} \\
\funt\funf\trip
    \ar[r]_{\funt\funf\effi} &
\funt\funf\triq
}},
\end{equation}
and we have the following lemma.
\begin{lemma}
\begin{enumerate}
\item The components of $\pseta_\effi$ are regular epimorphisms. $\pseta_\effi$
is a natural isomorphism whenever $\Phi$ commutes with $\exists$ along
diagonals.
\item The components of $\pseta_{\funf\effi}$ are monomorphisms.
$\pseta_{\funf\effi}$ is a natural isomorphism whenever $\Phi$ commutes with
existential quantification along projections.
\item For $C\in\flc$, 
\[
\xymatrix@C+8mm{
J_{\funf\triq}D_\triq F C
    \depi[r]^-{J_{\funf\triq}\pseta_{\effi,C}} &
J_{\funf\triq}\funf\effi D_\trip C 
    \mono[r]^-{\pseta_{\funf\effi,D_\trip C}} &
\funt\funf\effi J_{\funf\trip} D_\trip C
}
\]
is an epi-mono factorization of the unit constraint of
$\funt\funf\adj\funs\funu$ at $F$ and $C$.
\end{enumerate}
\end{lemma}
\begin{proof}
\emph{Ad 1.} This follows directly from the definition of $\pseta_\effi$
in~\eqref{eq_fs_unit_constr} and the characterization of regular
epimorphisms in $\funf\triq$ in
Lemma~\ref{lem:fp-qtop}.\ref{lem:fp-qtop:eff-quots}.

\emph{Ad 2.} The component of $\pseta_{\funf\effi}$ at $A\in\funf\trip$ is
given by $J_{\funf\triq}\funf\effi\eta_A$, where $\eta_A:A\emonepi\overline{A}$
is the unit of the reflection of $\funf\trip$ onto the coarse objects. It is
a monomorphism since $\eta_A$ is monic and $\funf\effi$ and $J_{\funf\triq}$
preserve finite limits. Since $\funt\funf\triq$ is balanced and
$J_{\funf\triq}$ preserves epis as a left adjoint, $\pseta_{\funf\effi,A}$ will
be an isomorphism as soon as $\funf\effi$ preserves epimorphisms. But
epimorphisms $f:(C,\rho)\to(D,\sigma)$ in $\funf\trip$ are characterized by the
judgement $\sigma(d,d)\vdash\exists c\qdot \rho(c,c)\wedge \sigma(fc,d)$ and
since this judgement only contains conjunction and existential quantification
along projections, epimorphisms will be preserved by $\funf\effi$ if $\Phi$
preserves those.

\emph{Ad 3.} This follows from 1 and 2, since $J_{\funf\triq}$ preserves
epimorphisms.
\end{proof}
The tripos transformation
$\famf(\wedge):\famf(\bool\times\bool)\to\famf(\bool)$ is an example
which commutes with $\exists$ along diagonals, but not along projections. The
double negation topology on the modified realizability
tripos~\cite{vanoosten1997modified} is an example that commutes with
$\exists$ along projections, but not along diagonals.

\section{Epilogue: Conclusion, observations and first
applications}\label{sec_epilogue}

Now that we finally know that the \tttc{} is part of a special biadjunction and
in particular a special functor, we want to see whether this helps us in any
way to understand what happens in the realizability constructions that we
mentioned in the beginning. Of special interest here are geometric morphisms
and subtoposes. 

Geometric morphisms of toposes arise naturally in our setting, because they are
just adjunctions in the 2-category $\cattop$. Adjunctions in $\cattrip$ are a
generalization
of what is usually known as geometric morphism of triposes
(normally~\cite{hjp80}, they are
only considered on a fixed base category). Now we know from
Lemma~\ref{lem_functors_between_regular_dc} that special functors preserve
adjunctions between \tworeg{} pre-equipments, and this gives a conceptual
explanation why the \tttc{} transforms geometric morphisms of triposes into
geometric morphisms of toposes.

Subtoposes coincide with idempotent monads in the 2-category $\cattop$, and
subtriposes coincide (almost, there is again the
question of fixed or varying base for the triposes) with idempotent monads in
$\cattrip$. Monads are not in general preserved by oplax functors
(nor by special functors), but nevertheless (as already described in
\cite{hjp80}) we can construct a local operator on a topos from a local
operator on the corresponding tripos. So how can we understand this? --- From
an abstract point of view, this works because every idempotent monad in
$\cattrip$ can be decomposed into an adjunction with isomorphic counit, so we
can apply the \tttc{} to the decomposition and reassemble the local operator in
$\cattop$ (idempotency is preserved by the second statement of
Lemma~\ref{lem_functors_between_regular_dc}).  For the case of and idempotent
monad $(T,\Theta):\trip\to\trip$ on a tripos $\trip:\flx\to\flc$, the second
tripos of the factorization has as base the full subcategory of $\flc$ on the
objects $C$ where $\eta_C:C\to TC$ is an isomorphism, and the fibers are the
subpreorders of $\trip_C$ on the $\Theta$-stable predicates.

We may also ask what happens to comonads. These are preserved by special
functors simply because they can be defined as oplax functors having as domain,
and oplax functors can be composed. However, it is possible that an
idempotent comonad on a tripos gives rise to a non-idempotent comonad on a
topos. An example of this is the comonad induced by the adjunction
$\delta\dashv\wedge:\catset(-,\bool\times\bool)\to\catset(-,\bool)$ from the
introduction. The induced comonad of triposes, which is given by the fiberwise
operation $(a,b)\mapsto(a\wedge b,a\wedge b)$ on $\bool$ is idempotent, but its
image
\[
 \catset\times\catset\to\catset\times\catset,\qquad (A,B)\mapsto (A\times
B,A\times B)
\]
under $\funt\circ\funs$ is not.

\appendix

\section{Non-extensional higher order intuitionistic logic}\label{sec_nehoil}

In this appendix, we present the logical system
mentioned in Section~\ref{suse_hol_in_triposes}. The derivation rules are
given in Table~\ref{table_non_ext_high_int}.
\begin{table}
\begin{center}
\begin{tabular}{cc}
\unary{\phantom{|}}{}{\Delta\csep\Gamma\vdash\top} &
\unary{\Delta\csep\Gamma\vdash\bot}{}{\Delta\csep\Gamma\vdash\varphi}\\[.5cm]
\binary{\Delta\csep\Gamma\vdash\varphi\!\!}{\!\!\Delta\csep\Gamma\vdash\psi}{}{
\Delta\csep\Gamma\vdash\varphi\wedge\psi} &
\unary{\Delta\csep\Gamma\vdash\varphi_1\wedge\varphi_2}{$i=1,2$}{
\Delta\csep\Gamma\vdash\varphi_i}\\[.5cm]
\unary{\Delta\csep\Gamma\vdash\varphi_i}{$i=1,2$}{
\Delta\csep\Gamma\vdash\varphi_1\vee\varphi_2} &
\ternary{\Delta\csep\Gamma\vdash\varphi\vee\psi\!\!}{\!\!\!\Delta\csep\Gamma,
\varphi\vdash\gamma\!\!\!}{\!\!\Delta\csep\Gamma,\psi\vdash\gamma}{}{
\Delta\csep\Gamma\vdash\gamma}\\[.5cm]
\unary{\Delta\csep\Gamma,\varphi\vdash\psi}{}{
\Delta\csep\Gamma\vdash\varphi\imp\psi} &
\binary{\Delta\csep\Gamma\vdash\varphi\imp\psi}{\Delta\csep\Gamma\vdash\varphi}
{}{\Delta\csep\Gamma\vdash\psi}\\[.5cm]
\unary{\Delta,x\vtp A\csep\Gamma\vdash\xi[x]}{}{\Delta\csep\Gamma\vdash\forall
x\qdot \xi[x]} &
\unary{\Delta\csep\Gamma\vdash\forall x\qdot
\xi[x]}{}{\Delta\csep\Gamma\vdash\xi[t]}\\[.5cm]
\unary{\Delta\csep\Gamma\vdash\xi[t]}{}{\Delta\csep\Gamma\vdash\exists x\qdot
\xi[x]} &
\binary{\Delta\csep\Gamma\vdash\exists x\qdot \xi[x]}{\Delta,x\vtp
A\csep\Gamma,\xi[x]\vdash\psi}{}{\Delta\csep\Gamma\vdash\psi}\\[.5cm]
\unary{{\phantom{|}}}{}{\Delta\csep\Gamma\vdash t=t} &
\binary{\Delta\csep\Theta[s,t]\vdash s=t\!\!\!}{\!\!\!\Delta,x\vtp
A\csep\Theta[x,x]\vdash \rho[x,x]}{}{\Delta\csep\Theta[s,t]\vdash
\rho[s,t]}\\[.7cm]
\end{tabular}
\unary{}{}{\Delta\csep\Gamma\vdash \exists m:\ptype A\; \forall x\in A\qdot
x\in m\leftrightarrow \xi[x]}\\[4mm]
 \unary{}{}{\Delta,x\vtp A,y\vtp B\csep\Gamma\vdash \exists ! z\vtp A\times
B\qdot \pi_1(z)=x\wedge\pi_2(z)=y}
\end{center}
$\Delta = x_1\vtp A_1,\dots,x_n\vtp A_n$ is a context of variables,\\
$\Gamma$ is a list of formulas in context $\Delta$,\\
$\varphi,\varphi_1,\varphi_2,\psi,\gamma$ are formulas in context $\Delta$,\\
$\xi[x]$ is a formula in context $(\Delta, x\vtp A)$,\\
$\Theta[x,y]$ is a list of formulas in context $(\Delta, x\vtp A, y\vtp A)$,
\\
$\rho[x,y]$ is a formula in context $(\Delta, x\vtp A, y\vtp A)$, and \\
$s,t:A$ are terms in context $\Delta$.
\caption{Deduction rules of non-extensional higher order intuitionistic
logic.}\label{table_non_ext_high_int}
\end{table}
Observe that the variable conditions for quantification and equality, that are
normally perceived as somewhat disturbing, arise naturally in the presence of
explicit variable contexts. The equality rules are equivalent to the
traditional ones, but are meant to resemble the rules for existential
quantification since from the categorical point of view, equality is
existential quantification along diagonals%

This system is called non-extensional since there is no axiom or rule that
states that two sets ($=$ individuals of power type) are equal whenever they
have the same elements.

\subsection*{Acknowledgements}

This article is partially based on my master thesis, which I wrote under
supervision of
Thomas Streicher. I want to thank him for introducing me to categorical logic
in general and in particular for his support on this project that manifested in
countless discussions.

I want  to thank Ignacio Fernando Lopez, John Baez and Martin Hyland for
pointers to related work in connection with equipments, and Bas
Spitters who made me aware of Maietti and Rosolini's work~\cite{rosmai08}.

Furthermore, thanks to Dorette Pronk for discussions on double categories at
the \emph{International Category Theory Conference CT2010} in Genoa, and to
Pino Rosolini for some helpful comments at the same
conference.

Special thanks to Dominic Verity for providing me a copy of his thesis, and
for very friendly and helpful feedback to my questions via mail.

Finally, I want thank to Paul-André Melliès for encouraging me to write
this article, and for support and advice in the process of writing.

\bibliographystyle{plain}
\bibliography{bib}
\end{document}